\documentclass[reqno,a4paper]{amsart}

\usepackage{amsmath}%
\usepackage{amsfonts}%
\usepackage{amssymb}%
\usepackage{graphicx}
\usepackage{faktor}
\usepackage{xfrac}
\usepackage{xcolor}
\usepackage{verbatim}
\usepackage{mathtools}
\usepackage{enumerate}
\usepackage{ stmaryrd }
\usepackage{cite}
\usepackage{tikz}
\usetikzlibrary{arrows}
\usepackage{tikz-cd}
\usepackage{hyperref}

\newtheorem{theorem}{Theorem}[section]
\theoremstyle{plain}
\newtheorem{corollary}[theorem]{Corollary}
\newtheorem{lemma}[theorem]{Lemma}
\newtheorem{proposition}[theorem]{Proposition}

\theoremstyle{definition}
\newtheorem{definition}[theorem]{Definition}
\newtheorem{example}[theorem]{Example}
\newtheorem{remark}[theorem]{Remark}
\numberwithin{equation}{section}

\def\quotient#1#2{%
    \raise1ex\hbox{$#1$}\Big/\lower1ex\hbox{$#2$}%
}

\DeclareMathOperator{\sgn}{sgn} 
\DeclareMathOperator{\Ker}{\mathrm{Ker}} 
\DeclareMathOperator{\Coker}{\mathrm{Coker}} 
\DeclareMathOperator{\Image}{\mathrm{Im}} 
\DeclareMathOperator{\Disc}{Disc} 

\newcommand{\uv}[1]{``{#1}"}

\newcommand{\Z}{\mathbb{Z}}
\newcommand{\Q}{\mathbb{Q}}

\newcommand{\abs}[1]{\left|{#1}\right|} 

\newcommand{\ii}{\mathrm{i}}

\newcommand{\f}[1]{\ensuremath{#1}} 
\newcommand{\OK}[1]{\ensuremath{\mathcal{O}_{#1}}} 
\newcommand{\ve}{\varepsilon}
\newcommand{\cjg}[1]{\overline{#1}} 

\newcommand{\chv}[1]{ \sgn(\sigma_1({#1})), \dots,  \sgn(\sigma_r({#1}))} 
\newcommand{\barsgn}[1]{\underline{\sgn}\left(#1\right)} 
\newcommand{\norm}[2]{\mathcal{N}_{#1}\left(#2\right)} 

\newcommand{\OISet}[1]{\mathcal{I}^{o}_{#1}} 
\newcommand{\OCl}[1]{\mathcal{C\!\ell}^{o}_{#1}} 

\newcommand{\OrelISet}[1]{\mathcal{I}^{o}_{#1}} 
\newcommand{\OrelPSet}[1]{\mathcal{P}^{o}_{#1}} 
\newcommand{\OrelCl}[1]{\mathcal{C\!\ell}^{o}_{#1}} 

\newcommand{\ISet}[1]{\mathcal{I}_{#1}} 
\newcommand{\PSet}[1]{\mathcal{P}_{#1}} 
\newcommand{\PPlusSet}[1]{\mathcal{P}^{+}_{#1}} 
\newcommand{\Cl}[1]{\mathcal{C\!\ell}_{#1}} 
\newcommand{\ClPlus}[1]{\mathcal{C\!\ell}^{+}_{#1}} 

\newcommand{\ClImag}[1]{\mathcal{C\!\ell}^{i}_{#1}} 
\newcommand{\Gr}[1]{\mathcal{G}_{#1}}

\newcommand{\USet}[1]{\mathcal{U}_{#1}} 
\newcommand{\UPlusSet}[1]{\mathcal{U}^{+}_{#1}} 

\newcommand{\OI}[2]{\left(#1; #2_1, \dots, #2_r \right)} 

\newcommand{\OrelI}[2]{\left(#1; #2_1, \dots, #2_r \right)} 
\newcommand{\OrelIbasis}[2]{\left( \left[#1\right]; \barsgn{\det #2} \right)} 
\newcommand{\OrelP}[1]{\left( \left(#1\right); \barsgn{ #1\cjg{#1} } \right) } 
\newcommand{\OrelPnorm}[2]{\left( \left(#1\right); \barsgn{ \norm{#2}{#1} } \right) } 

\newcommand{\pmr}[1]{\left\langle \pm 1 \right\rangle ^{#1}} 

\newcommand{\DSet}{\mathcal{D}} 
\newcommand{\QFSet}{\mathcal{Q}_\DSet}

\begin{document}
\title[Composition of Quadratic Forms]{Composition of Binary Quadratic Forms \\over Number Fields}
\author{Krist\'{y}na Zemkov\'{a}}

\address{Charles University, Faculty of Mathematics and Physics, Department of Algebra,
Sokolovsk\'{a} 83, 18600 Praha 8, Czech Republic}
\address{Fakult\"at f\"ur Mathematik, Technische Universit\"at Dortmund, D-44221 Dortmund,
Germany}
\email{zemk.kr@gmail.com}%

\thanks{The author was supported by the project SVV-2017-260456 and by the Charles University, project GA UK No.\ 1298218.}

\date{\today}

\begin{abstract}
In this article, the standard correspondence between the ideal class group of a quadratic number field and the equivalence classes of binary quadratic forms of given discriminant is generalized to any base number field of narrow class number one. The article contains an explicit description of the correspondence. In the case of totally negative discriminants, equivalent conditions are given for a binary quadratic form to be totally positive definite.
\end{abstract}
\maketitle

\section{Introduction}

There are two classical definitions of composition of binary quadratic forms over $\Z$: the first one, using bilinear substitutions, was described by Gauss in his Disquisitiones Arithmeticae in 1801; the second one, based on so-called \uv{united forms}, is attributed to Gauss's student, Dirichlet. Yet a completely different approach was taken by Dedekind. In modern terms, his idea was to associate a binary quadratic form with an appropriate module; the composition of quadratic forms is then translated as module multiplication. 

Later on, all of these approaches were used to generalize composition of binary quadratic forms to more general base rings. In \cite{Butts-Estes}, Butts and Estes determined domains, in which the \uv{united-forms} composition holds. Kaplansky in \cite{Kaplansky} used the module multiplication approach to give a full description of composition of binary quadratic forms over any B\'{e}zout domain. The relationship between Gauss composition and \uv{united-form} composition over some classes of rings was studied by Butts and Dulin in \cite{Butts-Dulin}. Towber in \cite{Towber} used \uv{oriented} binary quadratic forms to describe a composition over an arbitrary commutative ring with a unit, such that 2 is not a zero divisor. This last remaining restriction on the base ring was removed by Kneser in \cite{Kneser}, where he considered a binary quadratic module $M$ as a module over its even Clifford algebra $C^+(M)$; from this point of view, composition corresponds to the tensor product over $C^+(M)$. 

It was not until the beginning of the 21st century that Bhargava, in his famous article \cite{BhargavaLawsI}, redefined Gauss composition, and discovered another 13 composition laws on other polynomials. The beauty of these composition laws is in using certain cubes of integers, and the result is based on a correspondence between integral cubes and triples of ideals. A few years later in 2011, Bhargava's student Wood generalized Bhargava's correspondence in \cite{Wood} and gave a complete statement of the relationship between binary quadratic forms and modules for quadratic algebras over any base ring. Later, in 2016, O'Dorney described the same kind of correspondence  over any Dedekind domain in \cite{ODorney2016}. 

It is worth noting that in all the aforementioned articles, if an equivalence of quadratic forms (in the classical sense) is defined, then only the equivalence given by matrices of determinant 1 is considered. This can be seen as somewhat unnatural, as the base ring may contain other units as well. In the case of the ring of integers of a number field, the most natural choice seems to be to consider any equivalence given by a matrix, determinant of which is a totally positive unit. Indeed, this approach was taken by Mastropietro in his thesis \cite{mastropietro}; he described a construction of the correspondence between the ideal class group and the equivalence classes of binary quadratic forms, but only for the case when the base field is a real quadratic number field of class number one, and the discriminant is totally negative.

The aim of this article is to take into account an equivalence of binary quadratic forms given by all totally positive units of the base number field and to develop, in such settings, a Dedekind-like correspondence between classes of forms and an ideal class group. As the base field, we consider an arbitrary number field of narrow class number one; this is equivalent to having class number one together with the existence of units of all signs. These conditions are necessary for our approach: First, the class number has to be one, for a free module basis of any fractional ideal to exist, and second, units of all signs are needed in order for such a basis to be able to have any orientation. This paper is based on the first part of the author's Master thesis \cite{KZthesis}. The results of this paper are further used in \cite{KZcubes} to generalize the composition of Bhargava's cubes to rings of integers of number fields of narrow class number one.

First, we focus on base fields with at least one real embedding. Refined definitions of the equivalence of quadratic forms and of the ideal class group are given in Subsections~\ref{subsec:QFs} and~\ref{subsec:OCl}; then Section~\ref{sec:Correspondence} deals with the correspondence itself. The main results of the article are  Theorem~\ref{Theorem:Bijection} together with Corollary~\ref{Cor:GroupStructure}. Afterwards, the results are applied in Section~\ref{sec:PosDef} to a study of positive definiteness of quadratic forms. At the end of the paper, Section~\ref{sec:TotImag} deals with the case when the base field is totally imaginary; the result is summarized in Theorem~\ref{Theorem:Bijection-Imag}.


\section{Preliminaries}\label{Section:Prel}
Throughout the whole paper, we fix a number field $K$ of narrow class number one\footnote{We use the standard notation $h^+(K)$ for the narrow class number of $K$ and $h(K)$ for the class number of $K$.}; this is equivalent to $K$ being of class number one and having units of all signs (for a reference, see \cite[Ch.~V,~(1.12)]{FrohlichTaylor}). We write $\OK{K}$ for the ring of algebraic integers of $K$; the group of units of this ring is denoted by $\USet{K}$, and $\UPlusSet{K}$ stands for its subgroup of totally positive units. Assume that $K$ has exactly $r$ embeddings into real numbers, and let $\sigma_1, \dots, \sigma_r$ be these embeddings. Moreover, suppose that $r\geq1$, i.e., that $K$ is \emph{not totally imaginary}.

Furthermore, we fix a relative quadratic extension $L$ of the number field $K$. Note that the Galois group $\mathrm{Gal}(L/K)$ has two elements; if $\tau$ is the nontrivial element of this group, then, for $\alpha\in L$, we write $\cjg{\alpha}$ instead of $\tau(\alpha)$.


\subsection{Ideals}\label{Subsec:PrelIdeals}

Since $h^+(K)=1$, the ring $\OK{L}$ is a free $\OK{K}$-module; hence $\OK{L}=[1, \Omega]_{\OK{K}}$ for a~suitable $\Omega\in\OK{L}$, and every (fractional) $\OK{L}$-ideal has a module basis of the form $[\alpha, \beta]_{\OK{K}}$ for some $\alpha, \beta \in L$ (see \cite[Cor.~p.~388]{narkiewicz2004elementary} and \cite[Prop.~2.24]{MilneANT}). In both cases, the index $\OK{K}$ will be usually omitted. As an algebraic integer over $\OK{K}$, $\Omega$ is a root of a monic quadratic polynomial $x^2+wx+z$ for some $w,z\in\OK{K}$; the other root is $\cjg{\Omega}$. Put $D_\Omega=w^2-4z$, and note that, without loss of generality,
\begin{equation}\label{Omega}
\Omega=\frac{-w+\sqrt{D_{\Omega}}}{2}, \qquad \cjg{\Omega}=\frac{-w-\sqrt{D_{\Omega}}}{2}.
\end{equation} 
Hence, $\f{L}=\f{K}\left(\sqrt{D_{\Omega}}\right)$, and  
\begin{equation}\label{DOmega}
D_{\Omega}=\left(\Omega-\cjg{\Omega}\right)^2.
\end{equation}
Moreover, the above choice of $\Omega$ and $\cjg{\Omega}$ determines the square root of $D_\Omega$ uniquely as 
\begin{equation}\label{sqrtDOmega}
\sqrt{D_\Omega}=\Omega-\cjg{\Omega}.
\end{equation} 
Also note that $\OK{L}\subseteq\left\{\frac{a}{2}+\frac{b}{2}\sqrt{D_{\Omega}} ~\big|~ a,b \in \OK{K}\right\}$.

One may expect the element $D_\Omega$ to be square-free (i.e., not divisible by $q^2$ for any $q\in\OK{K}\backslash\:\USet{K}$), but since $D_\Omega$ is the discriminant of a binary quadratic form, it may not always be the case. Hence, instead of that, we introduce the following definition of fundamental element: an element, which is \uv{almost square-free} and a quadratic residue modulo $4$ at the same time. 

\begin{definition}\label{Def:AlmostSquare-free}
An element $d$ of $\OK{K}$ is called \emph{fundamental} if $d$ is a quadratic residue modulo $4$ in $\OK{K}$, and for every $p\in\OK{K}\backslash\:\USet{K}$ such that $p^2\mid d$ the following holds: $p\mid 2$ and $\frac{d}{p^2}$ is not a quadratic residue modulo $4$ in $\OK{K}$.
\end{definition}

In the case $K=\Q$, this definition agrees with the one of the \emph{fundamental discriminant}. The following lemma shows that, from this point of view, $D_{\Omega}$ is \uv{a~fundamental discriminant over $K$}.

\begin{lemma} \label{Lemma:DOmega}
$D_{\Omega}$ is fundamental.
\end{lemma}

\begin{proof}
Obviously, $D_\Omega=w^2-4z$ is a quadratic residue modulo $4$.

Assume that there exists $p\in\OK{K}$ which is not a unit, and such that $p^2\mid D_{\Omega}$; set $D'=\frac{D_{\Omega}}{p^2}$. Since $\sqrt{D'}$ is a root of the polynomial $x^2-D'$, it is $\sqrt{D'}\in\OK{L}$; hence, there exist $a, b \in\OK{K}$ such that $\sqrt{D'}=a+b\Omega$ where $\Omega=\frac{-w+p\sqrt{D'}}{2}$. Comparing the coefficients at $\sqrt{D'}$, we get that $p$ must be a divisor of $2$ in $\OK{K}$. 

For contradiction, suppose that there exists $t\in\OK{K}$ such that $D' \equiv t^2 \pmod{4}$. We can find $m\in\OK{K}$ such that $D'=t^2-4m$; then the quadratic polynomial $x^2+tx+m$ has the discriminant equal to $D'$ and a root $\kappa=\frac{-t+\sqrt{D'}}{2}$, which is an element of $\OK{L}$. Hence, there exist $a', b' \in\OK{K}$ such that $\kappa=a'+b'\Omega$, i.e.,
$$\frac{-t+\sqrt{D'}}{2} = a'+b'\frac{-w+p\sqrt{D'}}{2}.$$
Comparing the coefficients at $\sqrt{D'}$, we obtain that $b'p=1$. But that is not possible, because $b'\in\OK{K}$, and $p$ is not a unit. Hence, we have found the desired contradiction.
\end{proof}

On the other hand, if we take $D\in\OK{K}$ such that $K\big(\sqrt{D}\big)=L$, then clearly $p^2D=q^2D_{\Omega}$ for some $p,q\in\OK{K}$. Furthermore, if $D$ is fundamental, then $\frac{p}{q}$ has to be a unit, because both $\frac{p^2D}{q^2}$ and $\frac{q^2D_{\Omega}}{p^2}$ are quadratic residues modulo 4. We have proved the following lemma.

\begin{lemma}\label{Lemma:PosDiscr}
Let $D$ be a fundamental element of $\OK{K}$ s.t. $K\big(\sqrt{D}\big)=L$. Then there exists $u\in\USet{K}$ such that $D=u^2D_{\Omega}$.
\end{lemma}

In the following, the word ``ideal'' will generally stand for a fractional ideal while to the usual meaning will be referred as to the ``integral ideal''. Consider an ideal $I=[\alpha,\beta]$ in $\OK{L}$. Then there exists a $2\times2$ matrix $M$ consisting of elements of \f{K} such that
$$
\begin{pmatrix}
	 \alpha \\
		 \beta \\ 
	\end{pmatrix}
	= M\cdot
	\begin{pmatrix}
		 1 \\
	 \Omega \\
	\end{pmatrix}.
$$
Then also
\begin{equation}
\begin{pmatrix}
		\cjg{\alpha} & \alpha \\
		\cjg{\beta} & \beta \\ 
	\end{pmatrix}
	= M\cdot
	\begin{pmatrix}
		1 & 1 \\
		\cjg{\Omega} & \Omega \\
	\end{pmatrix},
\end{equation}

and thus
\begin{equation}
\det M = \frac{\cjg{\alpha}\beta-\alpha\cjg{\beta}}{\Omega-\cjg{\Omega}}.
\end{equation}

The proofs of the following two lemmas are just direct computations.
\begin{lemma}\label{Lemma:Change_of_orientation}
Let $I=[\alpha, \beta]$ be an ideal, $M$ the same matrix as above, $[p\alpha+r\beta, q\alpha+s\beta]$ another $\OK{K}$-module basis of $I$, and $\widetilde{M}$ the matrix corresponding to this basis. Then $\det \widetilde{M}=(ps-qr)\det M$.
\end{lemma}

\begin{lemma} \label{Lemma:integrality_of_detM}
$\det M \in \f{K}$, and if $\alpha, \beta \in \OK{L}$, then $\det M \in \OK{K}$.
\end{lemma}

The \emph{relative norm} of an element $\alpha$ is defined as $\norm{L/K}{\alpha}=\alpha\overline{\alpha}$. If $I$ is a fractional $\OK{L}$-ideal, then the $\OK{K}$-ideal $\norm{L/K}{I}=\left(\norm{L/K}{\alpha}: \alpha\in I \right)$ is the \emph{relative norm of the ideal $I$}. Note that if $I$, $J$ are two $\OK{L}$-ideals such that $I\subseteq J$, then $\norm{L/K}{I}\subseteq\norm{L/K}{J}$. For the relative norm of a principal ideal clearly holds that $\norm{L/K}{(\alpha)}=\left(\norm{L/K}{\alpha} \right)$. In the general case, the ideal $\norm{L/K}{I}$ has to be principal as well, because $h(K)=1$ by the assumption; its generator can be written explicitly in terms of the $\OK{K}$-module basis of $I$, as the following lemma shows.

\begin{lemma} \label{Lemma:norm_ideal}
Let $I=[\alpha, \beta]$ be an ideal, and $M$ the same matrix as above. Then $\det M$ generates the $\OK{K}$-ideal $\norm{L/K}{I}$.
\end{lemma}

\begin{proof}
This is a well-known result holding for any finite Galois extension $\f{E}/\f{F}$ such that $h(F)=1$. The proof can be found, e.g., in \cite[Th. 1]{Mann}.
\end{proof}

Let us now determine what will be later recognized as the inverse class of an ideal. We need to start with a technical lemma, which will also turn out to be useful later in the proof of Proposition~\ref{Prop:Phi_Im}, as the elements here will be exactly the coefficients of the quadratic form obtained from the ideal $[\alpha, \beta]$.

\begin{lemma}\label{Lemma:IntegralCoeff}
Let $[\alpha, \beta]$ be an ideal, and $M$ the same matrix as above. Then 
$$\frac{\alpha\cjg{\alpha}}{\det M}, \frac{\beta\cjg{\beta}}{\det M}, \frac{\cjg{\alpha}\beta+\alpha\cjg{\beta}}{\det M}$$
are coprime elements of $\OK{K}$.
\end{lemma}

\begin{proof}
We start by proving that $\frac{\alpha\cjg{\alpha}}{\det M}, \frac{\beta\cjg{\beta}}{\det M}, \frac{\cjg{\alpha}\beta+\alpha\cjg{\beta}}{\det M}\in\OK{K}$. First, assume that $\alpha, \beta \in \OK{L}$. Then $\det M \in \OK{K}$ by Lemma~\ref{Lemma:integrality_of_detM}; hence we need to show that the elements $\alpha\cjg{\alpha}, \beta\cjg{\beta}, \cjg{\alpha}\beta+\alpha\cjg{\beta}$ are divisible by $\det M$ in $\OK{K}$. Since $\alpha \in [\alpha, \beta]$, there is $(\alpha)\subset [\alpha, \beta]$. It follows from Lemma~\ref{Lemma:norm_ideal} that $\left(\alpha\cjg{\alpha}\right)=\left(\norm{L/K}{\alpha}\right)\subset \norm{L/K}{[\alpha, \beta]}=(\det M)$, and therefore $\alpha\cjg{\alpha}$ is divisible by $\det M$ in $\OK{K}$. By the same argument, $\beta\cjg{\beta}$ is divisible by $\det M$ in $\OK{K}$. Similarly, $(\alpha+\beta) \subset [\alpha, \beta]$ implies that $\norm{L/K}{\alpha+\beta}$ is divisible by $\det M$. Since $\norm{L/K}{\alpha+\beta}=\alpha\cjg{\alpha}+\beta\cjg{\beta}+\cjg{\alpha}\beta+\alpha\cjg{\beta}$, we see that $\cjg{\alpha}\beta+\alpha\cjg{\beta}=\norm{L/K}{\alpha+\beta}-\alpha\cjg{\alpha}-\beta\cjg{\beta}$ is divisible by $\det M$ in $\OK{K}$ as well.

In the general case, we can find $k\in\OK{K}$ such that $k\alpha, k\beta\in\OK{L}$. Hence, we may apply the first part of the proof to the ideal $[k\alpha, k\beta]$. Since the corresponding determinant is $k^2\det M$; the terms $k^2$ cancel out in the fractions.

Denote $a=\frac{\alpha\cjg{\alpha}}{\det M}$, $b=\frac{\cjg{\alpha}\beta+\alpha\cjg{\beta}}{\det M}$, $c=\frac{\beta\cjg{\beta}}{\det M}$. To prove that $a,b,c$ are coprime, first note that $b^2-4ac = \big(\Omega-\cjg{\Omega}\big)^2=D_{\Omega}$. Therefore, if $a,b,c$ were divisible by an element $p$ in $\OK{K}\backslash\:\USet{K}$, then $\frac{D_{\Omega}}{p^2}$ would be a quadratic residue modulo 4, which is not possible by Lemma~\ref{Lemma:DOmega}.
\end{proof}

\begin{proposition}\label{Prop:InverseIdeals}
Let $[\alpha, \beta]$ be an ideal. Then 
$$[\alpha, \beta]\cdot\left[\cjg{\alpha}, -\cjg{\beta}\right]=(\det M)$$ 
as $\OK{L}$-ideals.
\end{proposition}

\begin{proof}
Denote 
$$I=[\alpha, \beta], \qquad J=\left[\frac{\cjg{\alpha}}{\det M}, \frac{-\cjg{\beta}}{\det M}\right].$$
We will prove that $IJ=[1, \Omega]$, which is equivalent to the statement of the proposition. Note that 
$$\norm{L/K}{J}=\frac{1}{(\det M)^2}\norm{L/K}{[\cjg{\alpha}, -\cjg{\beta}]}=\left(\frac{1}{\det M}\right);$$
therefore, $\norm{L/K}{IJ}=(1)$ by the multiplicativity of the norm. Since we have $\norm{L/K}{[1, \Omega]}=(1)$, to prove that $[1, \Omega]=IJ$, we only need to show that $1, \Omega \in IJ$ (see \cite[Cor. to Th. 1]{Mann}). As $IJ$ is an $\OK{L}$-ideal as well, it even suffices to prove that $1\in IJ$.

Clearly, 
$$IJ=\left[\frac{\alpha\cjg{\alpha}}{\det M}, \frac{-\alpha\cjg{\beta}}{\det M}, \frac{\cjg{\alpha}\beta}{\det M}, \frac{-\beta\cjg{\beta}}{\det M} \right]_{\OK{K}}.$$
By Lemma~\ref{Lemma:IntegralCoeff},
$$\gcd\left(\frac{\alpha\cjg{\alpha}}{\det M},\frac{\cjg{\alpha}\beta+\alpha\cjg{\beta}}{\det M}, \frac{\beta\cjg{\beta}}{\det M} \right)=1;$$
therefore, $1\in IJ$.
\end{proof}


\subsection{Quadratic forms} \label{subsec:QFs}
By binary quadratic forms over $K$ we understand homogeneous polynomials in two variables of degree 2 with coefficients in $\OK{K}$, i.e., $Q(x, y)=ax^2+bxy+cy^2$ with $a, b, c \in \OK{K}$. For abbreviation, we will refer to them as \emph{quadratic forms}. By $\Disc(Q)$ we denote the discriminant of the quadratic form $Q$, i.e., $\Disc(Q)=b^2-4ac$. Comparing to the case of quadratic forms over $\Q$, we need to slightly broaden the equivalence relation.

\begin{definition}
Two quadratic forms $Q(x,y)$ and $\widetilde{Q}(x,y)$ are \emph{equivalent}, denoted by $Q\sim\widetilde{Q}$, if there exist elements $p, q, r, s\in\OK{K}$ satisfying $ps-qr \in \UPlusSet{K}$ and a totally positive unit $u \in \UPlusSet{K}$ such that $\widetilde{Q}(x,y)=u Q(px+qy, rx+sy)$.
\end{definition}

Let $Q(x,y)=ax^2+bxy+cy^2$ and $\widetilde{Q}(x,y)=uQ(px+qy, rx+sy)=\widetilde{a}x^2+\widetilde{b}xy+\widetilde{c}y^2$ be equivalent quadratic forms. Then 
\begin{equation}\label{Eq:equiv_QF-koef}\begin{array}{rcl} 
	\widetilde{a} & = & u(ap^2+bpr+cr^2),\\
	\widetilde{b} & = & u(2apq+b(ps+qr)+2crs),\\
	\widetilde{c} & = & u(aq^2+bqs+cs^2), 
\end{array}\end{equation}
and
\begin{equation}\label{Eq:equiv_QF-disc}\begin{array}{rl} 
	\Disc({\widetilde{Q}})&=  \widetilde{b}^2-4\widetilde{a}\widetilde{c}=u^2(ps-qr)^2(b^2-4ac)\\
	\ & = u^2(ps-qr)^2\Disc(Q).
\end{array}\end{equation}
On the other hand, there is
\begin{equation}\label{Eq:equiv_QF-koef_inverse}\begin{array}{rcl} 
	a & = & \frac{1}{u(ps-qr)^2}(\widetilde{a}s^2-\widetilde{b}rs+\widetilde{c}r^2),\\
	b & = & \frac{1}{u(ps-qr)^2}(-2\widetilde{a}qs+\widetilde{b}(ps+qr)-2\widetilde{c}pr),\\
	c & = & \frac{1}{u(ps-qr)^2}(\widetilde{a}q^2-\widetilde{b}pq+\widetilde{c}p^2). 
\end{array}\end{equation}
We say that a quadratic form $Q(x,y)$ \emph{represents}  $m\in\OK{K}$ if there exist $x_0, y_0\in\OK{K}$ such that $Q(x_0,y_0)=m$. A quadratic form $Q(x,y)=ax^2+bxy+cy^2$ is called \emph{primitive} if $\gcd(a,b,c)\in\USet{K}$. We state two lemmas classical in the case of quadratic forms over $\Q$; we omit the proofs because they are analogous to that case.

\begin{lemma}\label{Lemma:QFrepr}
Equivalent quadratic forms represent the same elements of $\OK{K}$ up to the multiplication by a totally positive unit.
\end{lemma}
\begin{lemma}
Let $Q$ be a primitive quadratic form, and $\widetilde{Q}\sim Q$. Then $\widetilde{Q}$ is also primitive.
\end{lemma}

We are interested in quadratic forms of given discriminant, namely  $D_{\Omega}=\left(\Omega-\cjg{\Omega}\right)^2$. But from (\ref{Eq:equiv_QF-disc}) we can see that equivalent quadratic forms do not always have the same discriminant; their discriminants may differ from each other by a square of a totally positive unit. Therefore, we will consider all quadratic forms of discriminants belonging to the set
$$\DSet=\left\{u^2\left(\Omega-\cjg{\Omega}\right)^2 ~\big|~ u \in \UPlusSet{K}\right\}.$$
Note that all the elements of $\DSet$ are fundamental (by Lemma~\ref{Lemma:DOmega}). We will denote by $\QFSet$ the set of all primitive quadratic forms of discriminant in $\DSet$ modulo the equivalence relation described above:
$$\QFSet=\quotient{\left\{Q(x,y)=ax^2+bxy+cy^2 ~\big|~ a,b,c\in\OK{K},\ \gcd(a,b,c)\in\USet{K}, \ \Disc(Q)\in\DSet \right\}}{\sim}.$$

\begin{remark}\label{rem:equivQF}
Note that if $\Disc(Q)=u^2D_{\Omega}$ for $u\in\UPlusSet{K}$, then $\Disc\left(\frac{1}{u}Q \right)=D_{\Omega}$. Hence, in every class of $\QFSet$, there is a quadratic form of discriminant exactly $D_{\Omega}$.

If $K$ is totally real (and of narrow class number one), then every totally positive unit is square of a unit (for a reference, see \cite[Prop. 2.4]{Edgar-Mollin-Peterson}). Consider equivalent quadratic forms $Q$ and $Q'$, such that $\Disc(Q)=\Disc(Q')$, i.e., 
$$Q'(x, y)=\frac{1}{ps-qr}Q(px+qy,rx+sy)$$
for some $p,q,r,s \in \OK{K}$. 
If $u\in\USet{K}$ is such that $ps-qr=u^2$, then 
$$Q\left(\frac pux+\frac quy, \frac rux + \frac suy \right)=Q'(x,y)$$
gives the equivalence of $Q$ and $Q'$ with determinant $1$. Therefore, in this case, our notion could be simplified to quadratic forms of the discriminant exactly $D_{\Omega}$, and to the equivalence by the matrices of determinant $1$. But then a change of basis of an ideal may lead to obtaining inequivalent quadratic forms, as we will see in Subsection~\ref{Section:Ideals_to_forms} (see Remark~\ref{rem:QFmultUnit}). 
\end{remark}

Let us prove two lemmas, which will be useful later.

\begin{lemma}\label{Lemma:automorphs}
Let $\mu\in\USet{L}$, and let $Q(x,y)$ be a quadratic form with $\Disc(Q)\in\DSet$. Then there exist $p_0,q_0,r_0,s_0 \in\OK{K}$ such that $p_0s_0-q_0r_0=\mu\cjg{\mu}$ and $Q(x,y) = \frac{1}{p_0s_0-q_0r_0}Q(p_0x+q_0y, r_0x+s_0y)$.
\end{lemma}

\begin{proof}
Let $Q(x, y)= ax^2+ bxy+cy^2$, and assume that $\Disc(Q)=D_{\Omega}$ (multiply $Q$ by a suitable totally positive unit if needed). There exist $u, v \in\OK{K}$ such that $\mu=\frac{u}{2}+\sqrt{D_{\Omega}}\frac{v}{2}$; hence we have $\mu\cjg{\mu}=\left(\frac{u}{2}\right)^2-D_{\Omega}\left(\frac{v}{2}\right)^2$. From the condition $p_0s_0-q_0r_0=\mu\cjg{\mu}$ and by comparing the coefficients of the quadratic forms $(p_0s_0-q_0r_0)Q(x,y)$ and $Q(p_0x+q_0y, r_0x+s_0y)$, we get the following system of equations:

\begin{equation}\label{Eq:system1}\begin{array}{rcl}  
p_0s_0-q_0r_0 &=& \left(\frac{u}{2}\right)^2-D_{\Omega}\left(\frac{v}{2}\right)^2, \vspace{2mm} \\
(p_0s_0-q_0r_0)a &=& ap_0^2+bp_0r_0+cr_0^2, \vspace{2mm} \\
(p_0s_0-q_0r_0)b &=& 2ap_0q_0+b(p_0s_0+q_0r_0)+2cr_0s_0, \vspace{2mm} \\
(p_0s_0-q_0r_0)c &=& aq_0^2+bq_0s_0+cs_0^2.
\end{array}\end{equation}
One can check that
\begin{align*}
p_0&=\frac{u-bv}{2}, &
q_0&=-cv, &
r_0&=av, &
s_0&=\frac{u+bv}{2} 
\end{align*}
fulfill the system of equations (\ref{Eq:system1}). It remains to show that $p_0, q_0, r_0, s_0$ are elements of $\OK{K}$. This is obvious for $q_0$ and $r_0$; for $p_0$ and $s_0$ it follows then from the fact that $p_0+s_0=u$ and $p_0s_0=\mu\cjg\mu+q_0r_0$ are elements of $\OK{K}$.
\end{proof}

\begin{lemma} \label{Lemma:roots_of_equiv_QFs}
Let $Q(x,y)=ax^2+bxy+cy^2$ be a quadratic form, $p, q, r, s \in\OK{K}$ such that $ps-qr\in\UPlusSet{K}$. Consider the quadratic form $\widetilde{Q}(x,y)=Q(px+qy, rx+sy)=\widetilde{a}x^2+\widetilde{b}xy+\widetilde{c}y^2$ equivalent to $Q(x,y)$. Denote $D=\Disc(Q)$ and $\widetilde{D}=\Disc(\widetilde{Q})$. Then, under the assumption that $\sfrac{\sqrt{\widetilde{D}}}{\sqrt{D}}\in\UPlusSet{K}$, it holds 
$$\frac{p\frac{-\widetilde{b}+\sqrt{\widetilde{D}}}{2\widetilde{a}}+q}{r\frac{-\widetilde{b}+\sqrt{\widetilde{D}}}{2\widetilde{a}}+s}=\frac{-b+\sqrt{D}}{2a}.$$
\end{lemma}

\begin{proof}
As $\frac{-\widetilde{b}+\sqrt{\widetilde{D}}}{2\widetilde{a}}$ is a root of the quadratic polynomial $\widetilde{Q}(x, 1)=\widetilde{a}x^2+\widetilde{b}x+\widetilde{c}$, we can write

\begin{equation*}
\begin{array}{rcll}
\begin{pmatrix}
	\frac{-\widetilde{b}+\sqrt{\widetilde{D}}}{2\widetilde{a}} & 1 \\
\end{pmatrix}
& \widetilde{Q} &
\begin{pmatrix}
	\frac{-\widetilde{b}+\sqrt{\widetilde{D}}}{2\widetilde{a}} \vspace{1mm} \\
	1 \\
\end{pmatrix}
& =0 \vspace{2mm}\\

\begin{pmatrix}
	\frac{-\widetilde{b}+\sqrt{\widetilde{D}}}{2\widetilde{a}} & 1 \\
\end{pmatrix}
\begin{pmatrix}
	p & r\\
	q & s\\
\end{pmatrix}
& Q &
\begin{pmatrix}
	p & q\\
	r & s\\
\end{pmatrix}
\begin{pmatrix}
	\frac{-\widetilde{b}+\sqrt{\widetilde{D}}}{2\widetilde{a}} \vspace{1mm}\\
	1 \\
\end{pmatrix}
& =0 \vspace{2mm}\\

\begin{pmatrix}
	p\frac{-\widetilde{b}+\sqrt{\widetilde{D}}}{2\widetilde{a}}+q & r\frac{-\widetilde{b}+\sqrt{\widetilde{D}}}{2\widetilde{a}}+s \\
\end{pmatrix}
& Q &
\begin{pmatrix}
	p\frac{-\widetilde{b}+\sqrt{\widetilde{D}}}{2\widetilde{a}}+q \vspace{1mm}\\
	r\frac{-\widetilde{b}+\sqrt{\widetilde{D}}}{2\widetilde{a}}+s \\
\end{pmatrix}
& =0 \vspace{2mm}\\

\begin{pmatrix}
	\frac{p\frac{-\widetilde{b}+\sqrt{\widetilde{D}}}{2\widetilde{a}}+q}{r\frac{-\widetilde{b}+\sqrt{\widetilde{D}}}{2\widetilde{a}}+s} & 1 \\
\end{pmatrix}
& Q &
\begin{pmatrix}
	\frac{p\frac{-\widetilde{b}+\sqrt{\widetilde{D}}}{2\widetilde{a}}+q}{r\frac{-\widetilde{b}+\sqrt{\widetilde{D}}}{2\widetilde{a}}+s} \vspace{1mm}\\
	1 \\
\end{pmatrix}
& =0 \\
\end{array}
\end{equation*}

In particular, $\frac{p\frac{-\widetilde{b}+\sqrt{\widetilde{D}}}{2\widetilde{a}}+q}{r\frac{-\widetilde{b}+\sqrt{\widetilde{D}}}{2\widetilde{a}}+s}$ has to be one of the roots of $Q(x, 1)$; either $\frac{-b+\sqrt{D}}{2a}$ or $\frac{-b-\sqrt{D}}{2a}$. Using that $\widetilde{D}=\widetilde{b}^2-4\widetilde{a}\widetilde{c}$, it is 
\[
\frac{p\frac{-\widetilde{b}+\sqrt{\widetilde{D}}}{2\widetilde{a}}+q}{r\frac{-\widetilde{b}+\sqrt{\widetilde{D}}}{2\widetilde{a}}+s}
= \frac{\left(2\widetilde{a}q-\widetilde{b}p+p\sqrt{\widetilde{D}}\right)\left(2\widetilde{a}s-\widetilde{b}r-r\sqrt{\widetilde{D}}\right)}{4\widetilde{a}\left(\widetilde{a}s^2-\widetilde{b}rs+\widetilde{c}r^2\right)}.
\]
Using the expression  (\ref{Eq:equiv_QF-disc}) with $u=1$ and applying the assumption that $\sqrt{\widetilde{D}}$ and $\sqrt{D}$ differ by a~totally positive unit, we get that $\sqrt{\widetilde{D}}=(ps-qr)\sqrt{D}$. Together with the first expression in (\ref{Eq:equiv_QF-koef_inverse}), we get from the above equality that the coefficient in front of $\sqrt{D}$ is $\frac{1}{2a}$ as claimed.
\end{proof}


\subsection{Relative Oriented Class Group}\label{subsec:OCl}

In the traditional correspondence, there are binary quadratic forms on one side, and the class group $\Cl{L}=\sfrac{\ISet{L}}{\PSet{L}}$, or the narrow class group $\ClPlus{L}=\sfrac{\ISet{L}}{\PPlusSet{L}}$, on the other side. Since we are working with number fields of higher degrees, the situation is a~bit more complicated. Inspired by Bhargava's definition of the class group, which considers the orientation of the bases of the ideals, we define the \emph{relative oriented class group} with respect to the extension $L/K$. Compared to the rational numbers, our base field $K$ has $r$ real embeddings; therefore, instead of one sign, we consider $r$ signs: one for every real embedding.  Later, in Section~\ref{sec:PosDef}, we will see that these signs are closely connected to the positive definiteness of the corresponding quadratic forms.

In the following, we will work with tuples of the form $\OrelI{I}{\ve}$ where $I$ is a fractional $\OK{L}$-ideal and $\ve_i\in\left\{\pm1\right\}$ for $i=1,\dots,r$. Since any $\OK{L}$-ideal can be described by an $\OK{K}$-module basis as $I=[\alpha, \beta]$, we can assign the values to the $\ve_i$'s by looking at the orientation of this basis at each real embedding: For every $i=1,\dots,r$ we set $\ve_i=\sgn(\sigma_i(\det M))$ where (as before)
\[\begin{pmatrix}
		\cjg{\alpha} & \alpha \\
		\cjg{\beta} & \beta \\ 
	\end{pmatrix}
	= M\cdot
	\begin{pmatrix}
		1 & 1 \\
		\cjg{\Omega} & \Omega \\
	\end{pmatrix},
\]
i.e., $\det M = \frac{\cjg{\alpha}\beta-\alpha\cjg{\beta}}{\Omega-\cjg{\Omega}}$. To shorten the notation, we set $\underline{\sgn}(a)=(\chv{a})$ for any $a\in K$. So we get the tuple 
\[\left( [\alpha, \beta]; \chv{ \det M} \right)=\OrelIbasis{\alpha,\beta}{M}.\]

In the special case of a principal ideal $(\gamma)$, the $\OK{K}$-module basis is $[\gamma, \gamma\Omega]$. Thus, the determinant of the corresponding matrix (as above) is $\gamma\cjg\gamma=\norm{L/K}{\gamma}$, and we get the tuple
\[\OrelP{\gamma}.\]

The following lemma shows that, under our assumption that $K$ is of narrow class number one, we can find a basis of any ideal $I$ of any prescribed orientation. With this view in mind, we will slightly abuse the terminology and call both of the tuples $\OrelI{I}{\ve}$ and $\OrelIbasis{\alpha,\beta}{M}$ an \emph{oriented ideal}.

\begin{lemma}
Let $I$ be a fractional $\OK{L}$-ideal and $\ve_i\in\left\{\pm1\right\}$ for $i=1,\dots,r$. Then there exists a basis $[\alpha, \beta]$ of $I$ such that $\barsgn{\det M}=(\ve_1, \dots, \ve_r)$, where $\det M = \frac{\cjg{\alpha}\beta-\alpha\cjg{\beta}}{\Omega-\cjg{\Omega}}$.
\end{lemma}

\begin{proof}
Consider any basis $[\alpha, \beta]$ of the ideal $I$, and multiply $\alpha$ by a unit $u\in\USet{K}$ with the appropriate $\barsgn{u}$ to get the desired orientation; such a unit exists by the assumption that the narrow class number of $\f{K}$ is one.
\end{proof}

For two oriented ideals $\OI{I}{\ve}$ and $\OI{J}{\delta}$, we can define its multiple componentwise as
$$\OI{I}{\ve}\cdot\OI{J}{\delta} = \left(IJ; \ve_1\delta_1, \dots, \ve_r\delta_r \right).$$

If these two ideals are $\OrelIbasis{\alpha,\beta}{M}$ and $\OrelIbasis{\alpha',\beta'}{M'}$, then their multiple is the ideal $[\alpha, \beta] \cdot [\alpha', \beta']$ with the orientation $\barsgn{\det M \cdot \det M'}$. The existence of a basis of the ideal $[\alpha, \beta] \cdot [\alpha', \beta']$ with such an orientation is guaranteed by the previous lemma. Note that here we need the condition that the field $K$ is of narrow class number one; otherwise, there may not exist any basis of the multiple of two ideals with the required orientation. 

If we multiply the oriented ideals  $\OrelIbasis{\alpha,\beta}{M}$ and $\OrelP{\gamma}$, one possible basis of the product is $[\gamma\alpha, \gamma\beta]$. The orientation of this basis is given by 
$$\begin{pmatrix}
		\cjg{\gamma\alpha} & \gamma\alpha \\
		\cjg{\gamma\beta} & \gamma\beta \\
	\end{pmatrix}
	= \widetilde{M}\cdot
	\begin{pmatrix}
		1 & 1 \\
		\cjg{\Omega} & \Omega \\
	\end{pmatrix},
$$
hence $\det\widetilde{M}=\gamma\cjg{\gamma}\det M$, i.e., we guessed a basis of the right orientation. Therefore,
$$\OrelP{\gamma}\cdot\OrelIbasis{\alpha, \beta}{M}=\OrelIbasis{\gamma\alpha, \gamma\beta}{\widetilde{M}}.$$

\smallskip

We can finally define the group which will stay on the other side of our desired correspondence.

\begin{definition}
Set
\begin{align*}
\OrelISet{L/K}&=\left\{ \OrelI{I}{\ve} ~\left|~ I \text{ a fractional $\OK{L}$-ideal}, \ve_i\in\left\{\pm1\right\}, i=1,\dots,r  \right\} \right., \vspace{1mm}\\ 
\OrelPSet{L/K}&=\left\{ \OrelPnorm{\gamma}{L/K} ~\left|~ \gamma \in \f{L}, \gamma\neq0  \right\} \right.. \footnotemark
\vspace{1mm} 
\end{align*}
\footnotetext{Here we use $\norm{L/K}{\gamma}$ rather than $\gamma\cjg\gamma$ for the matter of possible generalization to other than quadratic extensions; see Remark~\ref{Rem:ClassGroupGeneralization}.}
Then  the \emph{relative oriented class group} of the field extension $\f{L}/\f{K}$ is defined as
$$\OrelCl{L/K}=\quotient{\OrelISet{L/K}}{\OrelPSet{L/K}}.$$
If two oriented ideals $\OI{I}{\ve}$ and $\OI{J}{\delta}$ lie in the same class of $\OrelCl{L/K}$, we say that they are \emph{equivalent}, and we write $\OI{I}{\ve}\sim\OI{J}{\delta}$.
\end{definition}

With the multiplication defined above, $\OrelISet{L/K}$ is an abelian group and $\OrelPSet{L/K}$ is its subgroup. Therefore, the group $\OrelCl{L/K}$ is well defined. 

\begin{lemma}\label{Lemma:GroupElements}
The identity element of the group $\OrelCl{L/K}$ is $\left([1,\Omega]; +1, \dots, +1\right)$, and the inverse to $\OrelIbasis{\alpha,\beta}{M}$ is  $\OrelIbasis{\cjg{\alpha},-\cjg{\beta}}{M}$ (taking all of the oriented ideals as representatives of classes in $\OrelCl{L/K}$).
\end{lemma}

\begin{proof}
The orientation of the ideal $[1, \Omega]$ is $(+1, \dots, +1)$, because $M$ is in this case the unit matrix. Hence, the oriented ideal $\left([1,\Omega]; +1, \dots, +1\right)$ is a representative of the identity element of the group $\OrelCl{L/K}$.

We know from Proposition~\ref{Prop:InverseIdeals} that $[\alpha, \beta]\cdot\left[\cjg{\alpha}, -\cjg{\beta}\right]=(\det M)$; since the orientation of the product $[\alpha, \beta]\cdot\left[\cjg{\alpha}, -\cjg{\beta}\right]$ is $\barsgn{(\det M)^2}=\barsgn{\det M\cjg{\det M}}$, we even have that
$$\OrelIbasis{\alpha,\beta}{M}\cdot\OrelIbasis{\cjg{\alpha},-\cjg{\beta}}{M}=\OrelP{\det M} \in \OrelPSet{L/K}.$$
Thus, the oriented ideals  $\OrelIbasis{\alpha,\beta}{M}$ and $\OrelIbasis{\cjg{\alpha},-\cjg{\beta}}{M}$ represent the inverse classes in the group $\OrelCl{L/K}$.
\end{proof}

The following lemma states an easy but very useful observation; the proof is immediate. 

\begin{lemma} \label{Lemma:equivalent_ideals}
Two oriented ideals $\OI{I}{\ve}$ and $\OI{J}{\delta}$ are equivalent if and only if there exists $\gamma\in\f{L}$ such that $\gamma I = J$ and $\barsgn{\gamma\cjg{\gamma}}=\left({\ve_1}{\delta_1}, \dots, {\ve_r}{\delta_r}\right)$. Moreover, if $I=J$, then $\gamma\in\USet{L}$.
\end{lemma}

At the end of this subsection, we compare the relative oriented class group $\OrelCl{L/K}$ with the class group $\Cl{L}$ of the number field \f{L}. In the following, by $\OK{L}$ is understood the principal $\OK{L}$-ideal generated by a unit (hence the identity element in the group $\Cl{L}$), and by $\left\{\OK{L}\right\}$ the one-element group.

\begin{proposition}\label{Prop:OClDescription}
Denote $H=\left\{\barsgn{\norm{L/K}{\mu}} ~|~ \mu\in\USet{L}\right\}$ and $G$ the subgroup of $\OrelCl{L/K}$ isomorphic to $\left\{\OK{L}\right\}\times\sfrac{\pmr{r}}{H}$ (more precisely, $G$ is isomorphic to a subgroup of the group given by the representatives $(\OK{L};\pm1,\dots,\pm1)$). Then
$$\Cl{L} \simeq  \quotient{\OrelCl{L/K}}{G}\ .$$
\end{proposition}

\begin{proof}
Denote $\underline{\ve}=(\ve_1, \dots, \ve_r)$. Define maps $f$ and $g$:
$$\begin{array}{rccc}
f:&\ \left\{\OK{L}\right\}\times\pmr{r} & \longrightarrow & \OrelISet{L/K} \\
& \left(\OK{L}; \underline{\ve} \right) & \longmapsto & \left(\OK{L}; \underline{\ve} \right) \vspace{3mm}\\
g:&\ \OrelISet{L/K}&\longrightarrow&\ISet{L} \\
&\left(I;\underline{\ve}\right) & \longmapsto & I\\
\end{array}$$
Then $f$ is injective, $g$ surjective, and $\Ker{g}=\left\{\left(\OK{L};\underline{\ve}\right) ~|~ \underline{\ve}\in\pmr{r} \right\}=\Image{f}$. Consider restrictions $f'$ and $g'$ of these two maps:
$$\begin{array}{rccc}
f':&\ \left\{\OK{L}\right\}\times H & \longrightarrow & \OrelPSet{L/K} \\
& \left(\OK{L}; \barsgn{\norm{L/K}{\mu}} \right) & \longmapsto & \left((\mu); \barsgn{\norm{L/K}{\mu}} \right) \vspace{3mm}\\ 
g':&\ \OrelPSet{L/K}&\longrightarrow&\PSet{L} \\
& \OrelPnorm{\gamma}{L/K} & \longmapsto & (\gamma)\\
\end{array}$$
Note that $f'$ is indeed a restriction of $f$, as $(\mu)=\OK{L}$ for any $\mu\in\USet{L}$. Again, $f'$ is injective, $g'$ is surjective, and $\Ker{g'}=\left\{\OrelPnorm{\gamma}{L/K} ~|~ \gamma\in\USet{L}\right\}=\Image{f'}$. Hence, we obtain the following commutative diagram:

\begin{center}
\begin{tikzcd}
1 \arrow{r}
	& {\left\{\OK{L}\right\}\times\pmr{r}} \arrow{r}{f}
	& {\OrelISet{L/K}} \arrow{r}{g}
	& {\ISet{L}} \arrow{r}
	& 1 \\
1 \arrow{r}
	& {\left\{\OK{L}\right\}\times H} \arrow{r}{f'} \arrow[hookrightarrow]{u}{i_1}
	& {\OrelPSet{L/K}} \arrow{r}{g'} \arrow[hookrightarrow]{u}{i_2}
	& {\PSet{L}} \arrow{r} \arrow[hookrightarrow]{u}{i_3}
	& 1
\end{tikzcd}
\end{center}
Note that $\Coker{i_1}=\left\{\OK{L}\right\}\times\sfrac{\pmr{r}}{H}$, $\Coker{i_2}=\OrelCl{L/K}$, $\Coker{i_3}=\Cl{L/K}$, and $\Ker{i_3}=1$. Hence, by snake lemma, there is a short exact sequence:
\begin{center}
\begin{tikzcd}
1 \arrow{r}
	& \left\{\OK{L}\right\}\times\sfrac{\pmr{r}}{H} \arrow{r}{\iota}
	& {\OrelCl{L/K}} \arrow{r}
	& {\Cl{L}} \arrow{r}
	& 1
\end{tikzcd}
\end{center}
This sequence gives us the required isomorphism
$$\Cl{L} \simeq  \quotient{\OrelCl{L/K}}{G}\ $$
with $G=\iota\left(\left\{\OK{L}\right\}\times\sfrac{\pmr{r}}{H}\right)$.
\end{proof}

\begin{remark} \label{Rem:ClassGroupGeneralization}
The relative oriented class group could be defined for any finite Galois extension $E/F$ where the field $F$ is of narrow class number one; Proposition~\ref{Prop:OClDescription} would remain true without any modifications.
\end{remark}

\section{Correspondence between Ideals and Quadratic Forms}\label{sec:Correspondence}

\subsection{From Ideals to Quadratic Forms} \label{Section:Ideals_to_forms}

Let us define a map attaching a quadratic form to an oriented ideal.
$$\begin{array}{cccc}
\Phi: & \OrelCl{L/K} & \longrightarrow & \QFSet  \vspace{2mm}	\\
 & \OrelIbasis{\alpha,\beta}{M} & \longmapsto & \frac{1}{\det M} \norm{L/K}{\alpha x-\beta y }=\frac{\alpha\cjg{\alpha}x^2-(\cjg{\alpha}\beta+\alpha\cjg{\beta})xy+\beta\cjg{\beta}y^2}{\det M}\\
\end{array}$$
To be accurate, the map $\Phi$ goes between the classes of oriented ideals and classes of quadratic forms; we omit the classes to relax the notation. We have to check that this map is well defined. First, let us see that the image of this map lies in $\QFSet$.

\begin{proposition}\label{Prop:Phi_Im}
Let $\OrelIbasis{\alpha,\beta}{M}$ be a representative of a class in $\OrelCl{L/K}$, and denote by $Q_{\alpha,\beta}$ its image under the map $\Phi$. Then $Q_{\alpha,\beta}$ is a representative of a class of $\QFSet$.
\end{proposition}

\begin{proof}
We have
$$Q_{\alpha,\beta}(x, y) = \frac{\alpha\cjg{\alpha}x^2-(\cjg{\alpha}\beta+\alpha\cjg{\beta})xy+\beta\cjg{\beta}y^2}{\det M}.$$
Lemma~\ref{Lemma:IntegralCoeff} ensures both that the coefficients of $Q_{\alpha, \beta}$ are elements of $\OK{K}$, and that this quadratic form is primitive. One easily computes that $\Disc(Q_{\alpha,\beta})= \left(\Omega-\cjg{\Omega}\right)^2$, which is an element of $\DSet$. 
\end{proof}

To prove that the map $\Phi$ does not depend on the choice of the representative of the class in $\OrelCl{L/K}$, we start with a lemma, which connects units from the quadratic extension with the equivalence of quadratic forms:

\begin{lemma}\label{Lemma:Units_and_equivalence}
Let $Q$ be a primitive quadratic form, and $p,q,r,s \in \OK{K}$ be such that $ps-qr\in\USet{K}$. If there exists $\mu\in\USet{L}$ such that $\barsgn{\mu\cjg{\mu}}=\barsgn{ps-qr}$, then $Q(x,y) \sim \frac{1}{ps-qr}Q(px-qy, -rx+sy)$.
\end{lemma}

\begin{proof}
Use the elements $p_0, q_0, r_0, s_0$ from Lemma~\ref{Lemma:automorphs}: $p_0s_0-q_0r_0=\mu\cjg{\mu}$, and
$$Q(x,y)=\frac{1}{p_0s_0-q_0r_0}Q(p_0x+q_0y, r_0x+s_0y).$$
Note that $\barsgn{p_0s_0-q_0r_0}=\barsgn{\mu\cjg{\mu}}=\barsgn{ps-qr}$, and thus $\frac{ps-qr}{p_0s_0-q_0r_0}\in\UPlusSet{K}$. We can find the equivalence between the quadratic forms $\frac{1}{ps-qr}Q(px-qy, -rx+sy)$ and $\frac{1}{p_0s_0-q_0r_0}Q(p_0x+q_0y, r_0x+s_0y)$; this equivalence is obtained by the change of coordinates given by the matrix
$$
\begin{pmatrix}
	p_0 & q_0 \\
	r_0 & s_0 \\
\end{pmatrix}
\begin{pmatrix}
	p & -q \\
	-r & s \\
\end{pmatrix}^{-1},
$$
and by the multiplication by the totally positive unit $\frac{ps-qr}{p_0s_0-q_0r_0}$. Finally,
$$\frac{1}{ps-qr}Q(px-qy, -rx+sy) \sim \frac{1}{p_0s_0-q_0r_0}Q(p_0x+q_0y, r_0x+s_0y)=Q(x,y).$$
\end{proof}

\begin{proposition}\label{Prop:Phi_repr}
The map $\Phi$ does not depend on the choice of the representative $\OrelIbasis{\alpha,\beta}{M}$.
\end{proposition}

\begin{proof}
First, we would like to show that the definition of $\Phi$ is independent on the choice of the basis $[\alpha, \beta]$ of an ideal, i.e., that the quadratic form arising from the basis $[p\alpha+r\beta, q\alpha+s\beta]$, such that $p,q,r,s\in\OK{K}$, $ps-qr\in\USet{K}$, is equivalent to the one obtained from the basis $[\alpha, \beta]$; but this is not true in general, since the change of the basis  may change the orientation as well. Thus, we have to add an assumption that the oriented ideals obtained from these two bases are equivalent in $\OrelCl{L/K}$: Consider two oriented ideals 
$$\mathfrak{I}=\OrelIbasis{\alpha,\beta}{M},\qquad \widetilde{\mathfrak{I}}=\OrelIbasis{p\alpha+r\beta, q\alpha+s\beta}{\widetilde{M}},$$ such that $p,q,r,s\in\OK{K}$, $ps-qr\in\USet{K}$, and assume $\mathfrak{I}\sim\widetilde{\mathfrak{I}}$.
 Denote $Q(x,y)=\Phi(\mathfrak{I})$, $\widetilde{Q}(x,y)=\Phi(\widetilde{\mathfrak{I}})$; we need to prove that $Q\sim\widetilde{Q}$.  We have
\begin{align*}
&\widetilde{Q}(x,y) =\frac{\norm{L/K}{(p\alpha+r\beta)x-(q\alpha+s\beta)y}}{\det\widetilde{M}} \\
&= \frac{\scriptstyle{
			\big(p^2\alpha\cjg{\alpha}+pr(\cjg{\alpha}\beta+\alpha\cjg{\beta})+r^2\beta\cjg{\beta}\big)x^2 
			-\big(2pq\alpha\cjg{\alpha}+(ps+qr)(\cjg{\alpha}\beta+\alpha\cjg{\beta})+2rs\beta\cjg{\beta}\big)xy 
			+\big(q^2\alpha\cjg{\alpha}+qs(\cjg{\alpha}\beta+\alpha\cjg{\beta})+s^2\beta\cjg{\beta}\big)y^2 
			}}
			{\scriptstyle{(ps-qr)\det M}}\\
&=\frac{\scriptstyle{
				\alpha\cjg{\alpha}(px-qy)^2-\left(\cjg{\alpha}\beta+\alpha\cjg{\beta}\right)(px-qy)(-rx+sy)+\beta\cjg{\beta}(-rx+sy)^2}}
			{\scriptstyle{(ps-qr)\det M}}
=\frac{1}{ps-qr}Q(px-qy, -rx+sy),
\end{align*}
where $\det\widetilde{M}=(ps-qr)\det M$ by Lemma~\ref{Lemma:Change_of_orientation}. Since $\mathfrak{I}\sim\widetilde{\mathfrak{I}}$ and $[\alpha, \beta]=[p\alpha+r\beta, q\alpha+s\beta]$,  there exists $\mu\in\USet{L}$ by Lemma~\ref{Lemma:equivalent_ideals}, such that $\barsgn{\mu\cjg{\mu}}=\barsgn{ps-qr}$. Thus, the quadratic forms $Q(x,y)$ and $\widetilde{Q}(x,y)$ are equivalent by Lemma~\ref{Lemma:Units_and_equivalence}.

\smallskip

Now, let us consider any two equivalent oriented ideals; the equivalence is given by multiplication by a principal oriented ideal $\OrelP{\gamma}$. Thanks to the previous part of the proof, we may choose any bases we like. Thus, we consider the pair of oriented ideals $\OrelIbasis{\alpha,\beta}{M}$ and $\left([\gamma\alpha,\gamma\beta]; \barsgn{\gamma\cjg{\gamma}\det M}\right)$. The situation in this case is much easier, because the image of the oriented ideal $\left([\gamma\alpha,\gamma\beta]; \barsgn{\gamma\cjg{\gamma}\det M}\right)$ under the map $\Phi$ is 
\begin{multline*}
\frac{1}{\gamma\cjg{\gamma}\det M} \norm{L/K}{\gamma\alpha x-\gamma\beta y}
=\frac{\gamma\alpha\cjg{\gamma\alpha}x^2-(\cjg{\gamma\alpha}\gamma\beta+\gamma\alpha\cjg{\gamma\beta})xy+\gamma\beta\cjg{\gamma\beta}y^2}{\gamma\cjg{\gamma}\det M} \\
=\frac{\alpha\cjg{\alpha}x^2-(\cjg{\alpha}\beta+\alpha\cjg{\beta})xy+\beta\cjg{\beta}y^2}{\det M},
\end{multline*}
which is identical to the quadratic form obtained from $\OrelIbasis{\alpha,\beta}{M}$.
\end{proof}

\begin{remark}\label{rem:QFmultUnit}
The first part of the previous proof explains why we allowed multiplication by a totally positive unit in the definition of the equivalence of quadratic forms. Otherwise, the change of the basis from $[\alpha, \beta]$ to $[p\alpha+r\beta, q\alpha+s\beta]$ may result in inequivalent quadratic forms. (The totally positive unit is hidden in the use of Lemma~\ref{Lemma:Units_and_equivalence}.)
\end{remark}

\subsection{From Quadratic Forms to Ideals} \label{Section:Forms_to_ideals}

To get an oriented ideal from a quadratic form, define a map (again, we omit to write the classes)
$$\begin{array}{cccc}
\Psi: & \QFSet  & \longrightarrow & \OrelCl{L/K} \vspace{2mm}	\\
 & Q(x,y)=ax^2+bxy+cy^2 & \longmapsto & \left(\left[a, \frac{-b+\sqrt{\Disc(Q)}}{2}\right];\barsgn{a}\right)\\
\end{array}$$
where the square root is chosen such that $\sqrt{\Disc(Q)}=u\sqrt{D_{\Omega}}$ for some $u\in\UPlusSet{K}$; this choice is possible because $\Disc(Q)\in\QFSet$ and it is unique since only one unit from the pair $u$ and $-u$ can be totally positive. (Recall that the value of $\sqrt{D_\Omega}$ is fixed by the choice of $\Omega$ and $\cjg{\Omega}$, see (\ref{sqrtDOmega}).)

We have to show that this map is well defined.

\begin{proposition} \label{Prop:Psi_Im}
Let $Q(x,y)=ax^2+bxy+cy^2$ be a representative of a class in $\QFSet$. Then its image under the map $\Psi$ is an element of $\OISet{L/K}$.
\end{proposition}

\begin{proof}
Set $D=\Disc(Q)$, and let $D=u^2D_\Omega$ for a totally positive unit $u\in\UPlusSet{K}$. Recalling that $\Omega=\frac{-w+\sqrt{D_{\Omega}}}{2}$ by (\ref{Omega}), we have  
\begin{align*}
a\Omega&=\frac{b-uw}{2u}\cdot a + \frac{a}{u}\cdot\frac{-b+\sqrt{D}}{2},\\
\frac{-b+\sqrt{D}}{2}\Omega &= -\frac{c}{u}\cdot a-\frac{b+uw}{2u}\cdot\frac{-b+\sqrt{D}}{2}.
\end{align*}
By (\ref{Omega}) and (\ref{DOmega}), it is $w^2=(\Omega+\cjg{\Omega})^2=D_{\Omega}-4\Omega\cjg{\Omega}$, thus $u^2w^2=b^2-4ac-4u^2\Omega\cjg{\Omega}$, and it follows that $\frac{b+uw}{2u}\cdot\frac{b-uw}{2u}=-\Omega\cjg{\Omega}\in\OK{K}$. As $\frac{b+uw}{2u}+\frac{b-uw}{2u}=\frac{b}{u}\in\OK{K}$, too, it follows that $\frac{b+uw}{2u}, \frac{b-uw}{2u}\in\OK{K}$. Hence, $\left[a, \frac{-b+\sqrt{D}}{2}\right]$ is indeed an $\OK{L}$-ideal and we only need to compute the orientation of the basis. 

Let $M$ be a matrix such that
$$
\begin{pmatrix}
	a & a\\
	\frac{-b-\sqrt{D}}{2} & \frac{-b+\sqrt{D}}{2} \\
\end{pmatrix}
= M\cdot
\begin{pmatrix}
	1 & 1\\
	\cjg{\Omega} & \Omega \\
\end{pmatrix}.
$$
We need to prove that $\barsgn{\det M}=\barsgn{a}$. There is
$$ \det M= \frac{a\frac{-b+\sqrt{D}}{2} - a\frac{-b-\sqrt{D}}{2}}{\Omega-\cjg{\Omega}} = \frac{a\sqrt{D}}{\Omega-\cjg{\Omega}}=ua.$$
Since $u$ is totally positive, we have $\barsgn{\det M}=\barsgn{a}$.
\end{proof}

\begin{proposition}\label{Prop:Psi_repr}
The map $\Psi$ does not depend on the choice of the representative $Q(x,y)$.
\end{proposition}

\begin{proof}
Let $Q(x,y)=ax^2+bxy+cy^2$ be a quadratic form of discriminant $D\in\DSet$, and $v\in\UPlusSet{K}$. Then 
\[
\Psi(vQ(x,y))=\left(\left[va, \frac{-vb+\sqrt{v^2D}}{2}\right];\barsgn{va}\right) 
=\left(\left[a, \frac{-b+\sqrt{D}}{2}\right];\barsgn{a}\right) 
=\Psi(Q(x,y)). 
\]

\smallskip

Now, let $p, q, r, s \in\OK{K}$ be such that $ps-qr\in\UPlusSet{K}$, and consider the quadratic form $\widetilde{Q}(x,y)=Q(px+qy, rx+sy)=\widetilde{a}x^2+\widetilde{b}xy+\widetilde{c}y^2$ with $\Disc(\widetilde{Q})=\widetilde{D}$. We have
\begin{align*}
\Psi(Q(x,y)) &= \left(\left[a, \frac{-b+\sqrt{D}}{2}\right];\barsgn{a}\right) , \\
\Psi(\widetilde{Q}(x,y)) &= \left(\left[\widetilde{a}, \frac{-\widetilde{b}+\sqrt{\widetilde{D}}}{2}\right];\barsgn{\widetilde{a}}\right) ;	
\end{align*}
we need to show that these two oriented ideals are equivalent.

Let us first examine only how to come from the basis $\left[\widetilde{a}, \frac{-\widetilde{b}+\sqrt{\widetilde{D}}}{2}\right]$	 to the basis $\left[a, \frac{-b+\sqrt{D}}{2}\right]$; we will deal with the orientations afterwards. 

$$\begin{array}{ccccc}
\bigg[\widetilde{a}, \frac{-\widetilde{b}+\sqrt{\widetilde{D}}}{2}\bigg]
&\xrightarrow{\mathmakebox[1.5cm]{\cdot\frac{1}{\widetilde{a}}}}
&\bigg[1, \frac{-\widetilde{b}+\sqrt{\widetilde{D}}}{2\widetilde{a}}\bigg] 
&\xrightarrow{\mathmakebox[1.5cm]{\scriptsize{\begin{pmatrix}q&p\\s&r\end{pmatrix}}}}
&\bigg[p\frac{-\widetilde{b}+\sqrt{\widetilde{D}}}{2\widetilde{a}}+q, r\frac{-\widetilde{b}+\sqrt{\widetilde{D}}}{2\widetilde{a}}+s\bigg] \\
&\xrightarrow{\mathmakebox[1.5cm]{\cdot\scriptsize{\left(r\frac{-\widetilde{b}+\sqrt{\widetilde{D}}}{2\widetilde{a}}+s\right)^{-1}}}}
&\bigg[\frac{p\frac{-\widetilde{b}+\sqrt{\widetilde{D}}}{2\widetilde{a}}+q}{r\frac{-\widetilde{b}+\sqrt{\widetilde{D}}}{2\widetilde{a}}+s}, 1\bigg] 
&\stackrel{\text{\tiny{(Lemma~\ref{Lemma:roots_of_equiv_QFs}) }}}{=}
&\bigg[\frac{-b+\sqrt{D}}{2a}, 1\bigg] \\
&\xrightarrow{\mathmakebox[1.5cm]{\scriptsize{\begin{pmatrix}0&1\\1&0\end{pmatrix}}}}
&\bigg[1,\frac{-b+\sqrt{D}}{2a}\bigg]
&\xrightarrow{\mathmakebox[1.5cm]{\cdot a}}
&\bigg[a, \frac{-b+\sqrt{D}}{2}\bigg]
\end{array}$$

Recall that the multiplication of the basis by an element $\gamma$ change the orientation by $\barsgn{\gamma\cjg{\gamma}}$, and the transformation of the basis by a matrix $M$ change the orientation by $\barsgn{\det M}$. Since $a \in\f{K}$, there is $a\cjg{a}=a^2$, which is totally positive. Thus, the multiplication by $a$ does not change the orientation; the same holds for $\frac{1}{\widetilde{a}}\in\f{K}$. Also, both $-(ps-qr)$ and $-1$ are totally negative, and hence both the transformations by the matrices $\begin{pmatrix}q&p\\s&r\end{pmatrix}$ and $\begin{pmatrix}0&1\\1&0\end{pmatrix}$ change the orientation to the opposite one; together they does not affect the orientation. Therefore, the only impact might have the multiplication by $\left(r\frac{-\widetilde{b}+\sqrt{\widetilde{D}}}{2\widetilde{a}}+s\right)^{-1}$:
\begin{align*}
&\barsgn{\left(r\frac{-\widetilde{b}+\sqrt{\widetilde{D}}}{2\widetilde{a}}+s\right)^{-1}\cjg{\left(r\frac{-\widetilde{b}+\sqrt{\widetilde{D}}}{2\widetilde{a}}+s\right)^{-1}}}\\
&=\barsgn{\left(r\frac{-\widetilde{b}+\sqrt{\widetilde{D}}}{2\widetilde{a}}+s\right)\left(r\frac{-\widetilde{b}-\sqrt{\widetilde{D}}}{2\widetilde{a}}+s\right)}\\
&=\barsgn{\left(-r\widetilde{b}+2s\widetilde{a}\right)^2-r^2\widetilde{D}}
=\barsgn{r^2\widetilde{b}^2-4rs\widetilde{a}\widetilde{b}+4s^2\widetilde{a}^2-r^2(\widetilde{b}^2-4\widetilde{a}\widetilde{c})}\\
&=\barsgn{4\widetilde{a}(\widetilde{a}s^2-\widetilde{b}rs+\widetilde{c}r^2)}
\stackrel{\text{\tiny{(\ref{Eq:equiv_QF-koef_inverse}) }}}{=} \barsgn{4\widetilde{a}(ps-qr)^2a}
=\barsgn{\widetilde{a}a}
\end{align*}
(we used the facts that $4\widetilde{a}^2$ and $4(ps-qr)^2$ are totally positive). 
All the transformations changed the orientation from $\barsgn{\widetilde{a}}$ to $\barsgn{\widetilde{a}}\cdot\barsgn{\widetilde{a}a}=\barsgn{\widetilde{a}^2a}=\barsgn{a}$, and that is exactly the desired orientation. Hence, the oriented ideals $\left(\left[a, \frac{-b+\sqrt{D}}{2}\right];\barsgn{a}\right)$ and $\left(\left[\widetilde{a}, \frac{-\widetilde{b}+\sqrt{\widetilde{D}}}{2}\right];\barsgn{\widetilde{a}}\right)$ are equivalent.
\end{proof}


\subsection{Conclusion}
We are ready to prove that the maps $\Phi$ and $\Psi$ defined in the two previous subsections are mutually inverse bijections.

\begin{theorem}\label{Theorem:Bijection}
Let \f{K} be a number field of narrow class number one with at least one real embedding, and let $D$ be a fundamental element of $\OK{K}$. Set $L=K\big(\sqrt{D}\big)$, and $\DSet=\left\{u^2D ~|~ u\in\UPlusSet{K}\right\}$. We have a bijection 
$$\begin{array}{ccc}
\QFSet
&\stackrel{1:1}{\longleftrightarrow} 
& \OrelCl{L/K} \vspace{3mm}\\
Q(x, y)=ax^2+bxy+cy^2 & \stackrel{\Psi}{\longmapsto} & \left(\left[a, \frac{-b+\sqrt{\Disc(Q)}}{2}\right];\barsgn{a}\right) \vspace{2mm}\\
\frac{\alpha\cjg{\alpha}x^2-(\cjg{\alpha}\beta+\alpha\cjg{\beta})xy+\beta\cjg{\beta}y^2}{\frac{\cjg{\alpha}\beta-\alpha\cjg{\beta}}{\sqrt{D}}} 
&\stackrel{\Phi}{\longmapsfrom}
&\left(\left[\alpha, \beta\right];\barsgn{\frac{\cjg{\alpha}\beta-\alpha\cjg{\beta}}{\sqrt{D}}}\right)   
\end{array}$$
where $\sqrt{\Disc(Q)}$ and $\sqrt{D}$ are chosen such that $\sfrac{\sqrt{\Disc(Q)}}{\sqrt{D}}\in\UPlusSet{K}$.
\end{theorem}

\begin{proof}
Let $\Omega$ be such that $\OK{L}=[1, \Omega]$. By Lemma~\ref{Lemma:PosDiscr}, there is a unit $u\in\USet{K}$ (not necessarily totally positive) such that $D=u^2D_{\Omega}$; then we have $\OK{L}=[1, u\Omega]$ and $u^2D_{\Omega}=(u\Omega-\cjg{u\Omega})^2$. If we begin with the canonical basis $[1, u\Omega]$ of $\OK{L}$ instead of $[1, \Omega]$, we get the same results for $D$ in the place of $D_{\Omega}$. Therefore, without loss of generality, we may assume that $D=D_{\Omega}$ (and thus $\sqrt{D}=\Omega-\cjg{\Omega}$). 

Let 
$$\mathfrak{I}=\OrelIbasis{\alpha, \beta}{M}$$ 
be a representative of a class in $\OrelCl{L/K}$, and  
$$Q_{\alpha, \beta}(x,y)=\frac{\alpha\cjg{\alpha}x^2-(\cjg{\alpha}\beta+\alpha\cjg{\beta})xy+\beta\cjg{\beta}y^2}{\det M}$$ 
its image under the map $\Phi$; here we have $\Disc(Q_{\alpha,\beta})=(\Omega-\cjg{\Omega})^2$. If we now use the map $\Psi$, we obtain the oriented ideal
$$\mathfrak{I'}=\left( \left[\frac{\alpha\cjg{\alpha}}{\det M}, \frac{\cjg{\alpha}\beta+\alpha\cjg{\beta}}{2\det M}+\frac{\Omega-\cjg{\Omega}}{2} \right]; \barsgn{\frac{\alpha\cjg{\alpha}}{\det M}}\right).$$
Since $\det M = \frac{\cjg{\alpha}\beta-\alpha\cjg{\beta}}{\Omega-\cjg{\Omega}}$, there is $\frac{\cjg{\alpha}\beta+\alpha\cjg{\beta}}{2\det M}+\frac{\Omega-\cjg{\Omega}}{2} = \frac{(\cjg{\alpha}\beta+\alpha\cjg{\beta})+(\cjg{\alpha}\beta-\alpha\cjg{\beta})}{2\det M} = \frac{\cjg{\alpha}\beta}{\det M}$. Thus, 
$$\mathfrak{I'}=\left( \left[\frac{\alpha\cjg{\alpha}}{\det M}, \frac{\cjg{\alpha}\beta}{\det M}\right]; \barsgn{\frac{\alpha\cjg{\alpha}}{\det M}}\right).$$
If we multiply $\mathfrak{I'}$ by the principal oriented ideal $\left(\left(\frac{\det M}{\cjg{\alpha}}\right); \barsgn{\frac{\det M}{\cjg{\alpha}}\frac{\det M}{\alpha}} \right)$, we get exactly the ideal $\mathfrak{I}$. Therefore, $\mathfrak{I}\sim\mathfrak{I'}=\Psi\Phi(\mathfrak{I})$, and $\Psi\circ\Phi=\mathrm{id}_{\OrelCl{L/K}}$.

\smallskip

On the other hand, consider
$$Q(x,y)=ax^2+bxy+cy^2,$$
a representative of a class in $\QFSet$. Its image under the map $\Psi$ is the oriented ideal
$$\left(\left[a, \frac{-b+\sqrt{\Disc(Q)}}{2}\right];\barsgn{a}\right).$$
Using the map $\Phi$, we get a quadratic form
$$Q'(x, y)=\frac{a^2x^2+abxy+\frac{b^2-\Disc(Q)}{4}y^2}{\det M'},$$
where
$$
\begin{pmatrix}
	a & a\\
	\frac{-b-\sqrt{\Disc(Q)}}{2} & \frac{-b+\sqrt{\Disc(Q)}}{2} \\
\end{pmatrix}
= M'\cdot
\begin{pmatrix}
	1 & 1\\
	\cjg{\Omega} & \Omega \\
\end{pmatrix},
$$
and $\det M'=u'a$ for a unit $u'\in\UPlusSet{K}$ such that $\sqrt{\Disc(Q)}=u'(\Omega-\cjg{\Omega})$. Hence,
$$Q'(x,y)=\frac{1}{u'}(ax^2+bxy+cy^2),$$
and $Q\sim Q'=\Phi\Psi(Q)$. Therefore, $\Phi\circ\Psi=\mathrm{id}_{\QFSet}$.
\end{proof}

\begin{corollary}\label{Cor:GroupStructure}
$\QFSet$ carries a group structure arising from the multiplication of ideals in $K\big(\sqrt{D}\big)$. The identity element of this group is represented by the quadratic form $x^2-(\Omega+\cjg{\Omega})xy+\Omega\cjg{\Omega}y^2$, and the inverse element to $ax^2+bxy+cy^2$ is the quadratic form $ax^2-bxy+cy^2$.
\end{corollary}

\begin{proof}
The group structure of $\QFSet$ is given by the bijection with $\OrelCl{L/K}$ from Theorem~\ref{Theorem:Bijection}. 
The identity element is given as the image under the map $\Phi$ of the oriented ideal $\left([1,\Omega]; +1, \dots, +1\right)$ (which represents the identity element in $\OrelCl{L/K}$); thus, the identity element is $x^2-(\Omega+\cjg{\Omega})xy+\Omega\cjg{\Omega}y^2$. 

Consider the quadratic forms $ax^2+bxy+cy^2$ and $ax^2-bxy+cy^2$, and set $D=b^2-4ac$. The images of these two quadratic forms under the map $\Phi$ are the oriented ideals 
$$\left(\left[a, \frac{-b+\sqrt{D}}{2}\right];\barsgn{a}\right) \text{ and } \left(\left[a, \frac{b+\sqrt{D}}{2}\right];\barsgn{a}\right),$$
 which represent the mutually inverse classes of $\OrelCl{L/K}$ by Lemma~\ref{Lemma:GroupElements}. Hence, the quadratic forms $ax^2+bxy+cy^2$ and $ax^2-bxy+cy^2$ are inverse to each other.
\end{proof}

\subsection{Quadratic fields}

At the end of this section, let us look at the case of quadratic fields: Assume $K=\Q$, and  $L=\Q\big(\sqrt{D}\big)$. In this case, there exists only one real embedding of $K$, and that is the identity. Thus, Proposition~\ref{Prop:OClDescription} says that
$$
\Cl{L} \simeq  \quotient{\OrelCl{L/K}}{G}
 \quad \text{ with } \quad 
G\cong\left\{\OK{L}\right\}\times\sfrac{\pmr{}}{\left\{\sgn{\norm{L/K}{\mu}} ~|~ \mu\in\USet{L}\right\}}\ .
$$
To give more precise results, we need to distinguish three possible cases according to the sign of $D$ and the existence of a negative unit:

1. If $D<0$, then there is $\norm{L/K}{\gamma}>0$ for every $\gamma\in L$; therefore, the oriented ideals $\left([\alpha, \beta]; \sgn(\det M)\right)$ and $\left([-\alpha, \beta]; -\sgn(\det M)\right)$ cannot be equivalent, and it is easy to see that $\OCl{L/K}\simeq\Cl{L}\times\pmr{}$. Moreover, the quadratic forms $ax^2+bxy+cy^2$ and $-ax^2+bxy-cy^2$ cannot be equivalent either; if one of them is positive definite, then the other one is negative definite. This explains the factor $\pmr{}$ in the relative oriented class group because in the usual correspondence only positive definite forms are considered whenever $D<0$.

2. Let $D>0$, and assume that every unit has positive norm, i.e., for every $\mu\in\USet{L}$ there is $\norm{L/K}{\mu}=+1$. Consider the following surjective homomorphism:
$$\begin{array}{cccc}
f: & \OrelISet{L/K} & \longrightarrow  & \ClPlus{L} \vspace{1mm}\\
&(I;+1) & \longmapsto  & I\PPlusSet{L} \vspace{1mm}\\
 &(I;-1) & \longmapsto & \sqrt{D}I\PPlusSet{L}\\
\end{array}$$
Since $\Ker f=\OrelPSet{L/K}$, it follows that $\OrelCl{L/K}\simeq\ClPlus{L}$. Furthermore, it is well known that in this case is $\ClPlus{L}\simeq\Cl{L}\times\pmr{}$; hence, we have as well that $\OrelCl{L/K}\simeq\Cl{L}\times\pmr{}$.

3. Finally, assume that $D>0$, and that there exists a unit of negative norm, i.e., there exists $\mu_0\in\USet{L}$ such that $\norm{L/K}{\mu_0}=-1$; then $H=\left\{(+1), (-1) \right\}=\pmr{}$. Therefore, we get directly from Proposition~\ref{Prop:OClDescription} that $\OrelCl{L/K}\simeq\Cl{L}$. Since it is well known that in this case the groups $\Cl{L}$ and $\ClPlus{L}$ are isomorphic, we have as well that $\OrelCl{L/K}\simeq\ClPlus{L}$.

Let us summarize our observations into the following proposition:

\begin{proposition}
Let $\f{K}=\Q$, and let $\f{L}=\Q\left(\sqrt{D_\Omega}\right)$ for $D_\Omega$ a fundamental element of $\Z$.
\begin{enumerate}[1.]
	\item If $D_\Omega<0$, then $\OrelCl{L/\Q}\simeq\Cl{L}\times\pmr{}$.
	\item If $D_\Omega>0$ and $\mu\overline{\mu}=1$ for every $\mu\in\USet{L}$, then $\OrelCl{L/\Q}\simeq\Cl{L}\times\pmr{}\simeq\ClPlus{L}$.
	\item If $D_\Omega>0$ and there exists $\mu\in\USet{L}$ such that $\mu\overline{\mu}=-1$, then $\OrelCl{L/\Q}\simeq\Cl{L}\simeq\ClPlus{L}$.
\end{enumerate}
\end{proposition}

\begin{remark}
The proposition shows us that the relative oriented class group $\OrelCl{L/K}$ (and hence the correspondence between oriented ideals and quadratic forms) is a generalization of Bhargava's view to the classical correspondence; see \cite[Section 3.2]{BhargavaLawsI}.
\end{remark}


\section{Totally positive definite quadratic forms}\label{sec:PosDef}

A very interesting and well-studied class of quadratic forms are the totally positive definite ones, which form a natural generalization of sums of squares. As such, they have been studied for example in the context of representations of totally positive integers, e.g., in \cite{Blomer-Kala:rank, Blomer-Kala:n-ary, CLSTZ, Chan-Kim-Raghavan, Earn-Khos, Kala:n-ary, KTZ, Siegel}. Of course, a binary quadratic form can never be universal; nevertheless, our results may prove to be useful also in the study of quadratic forms of higher ranks.

Recall that we assume $K$ to have at least one real embedding. If $D_\Omega$ is totally negative (i.e., $\sigma_i(D_\Omega)<0$ for all $i=1, \dots, r$; this fact will be denoted by $D_\Omega\prec0$), then we can study totally positive definite quadratic forms. Recall that a quadratic form $Q(x,y)=ax^2+bxy+cy^2$ is called \emph{totally positive definite} if it holds that $\sigma_i(Q(x,y))=\sigma_i(a)x^2+\sigma_i(b)xy+\sigma_i(c)y^2$ is positive definite for all $i=1, \dots, r$. It is clear from the matrix notation that $Q(x,y)=ax^2+bxy+cy^2$ of discriminant in $\DSet=\{u^2D_\Omega~|~u\in\UPlusSet{K}\}$ is totally positive definite if and only if $a\succ0$ (i.e., $a$ is totally positive). Hence, the image under the map $\Psi$ of a totally positive definite quadratic form $Q(x,y)=ax^2+bxy+cy^2$ is the oriented ideal $\left( \left[a, \frac12\left(-b+\sqrt{\Disc(Q)}\right) \right]; +1, \dots, +1 \right)$. 

One may ask if it is possible to describe the totally positive definite quadratic forms in terms of the oriented ideals. We start with the following lemma, which claims that if $D_\Omega\prec0$, then all ideals in the same class of $\OrelCl{\f{L}/\f{K}}$ have the same orientation. 

\begin{lemma}
If $D_\Omega\prec0$ and $\OrelI{I}{\ve} \sim \OrelI{J}{\delta}$, then $\ve_i=\delta_i$ for all $i=1, \dots, r$.
\end{lemma}

\begin{proof}
By Lemma~\ref{Lemma:equivalent_ideals}, if $\OrelI{I}{\ve} \sim \OrelI{J}{\delta}$, then there exists $\gamma \in \f{L}$ such that $J=\gamma I$ and $\delta_i=\sgn(\sigma_i(\gamma\overline{\gamma}))\ve_i$ for all $i=1, \dots, r$. Since $D_\Omega\prec0$, there is $\gamma\overline{\gamma}\succ0$ for every $\gamma\in\f{L}$.  Therefore, $\sgn(\sigma_i(\gamma\overline{\gamma}))=+1$, and $\delta_i=\ve_i$ for all $i=1, \dots, r$.
\end{proof}

\begin{proposition}\label{QFsign}
Let $D_\Omega\prec0$, let $Q\in\QFSet$, and let $i\in\left\{1, \dots, r\right\}$. Then the following are equivalent:
\begin{enumerate}[(i)]
	\item $\sigma_i(Q)$ is positive definite, 
	\item $Q$ is the image under the map $\Phi$ of $\OrelIbasis{\alpha, \beta}{M}$ such that $\sigma_i(\det M)>0$, 
	\item $Q$ is the image under the map $\Phi$ of $\OrelIbasis{\alpha, \beta}{M}$ such that $\sigma_i\left(\Im\left(\frac{\beta}{\alpha}\right)\right)>0$, where $\Im\left(c_1+c_2\sqrt{D_{\Omega}}\right)=c_2$ for any $c_1, c_2 \in \f{K}$.
\end{enumerate}
\end{proposition}

\begin{proof}
Let $Q(x, y)=ax^2+bxy+cy^2$. Since $D_\Omega\prec0$, the positive definiteness of $\sigma_i(Q)$ is given by the sign of $\sigma_i(a)$. First, we will prove that $(i) \Leftrightarrow (ii)$: Recall that 
$$\Psi(Q)=\left( \left[a, \frac{-b+\sqrt{\Disc(Q)}}{2} \right]; \barsgn{a} \right).$$
If $\mathfrak{I}=\OrelIbasis{\alpha,\beta}{M}$ is an oriented ideal such that $\Phi(\mathfrak{I})=Q(x,y)$, then $\Psi\Phi(\mathfrak{I})=\Psi(Q)$. Hence, since $\mathfrak{I} \sim \Psi\Phi(\mathfrak{I})$, there is
$$\OrelIbasis{\alpha,\beta}{M} \sim  \left( \left[a, \frac{-b+\sqrt{\Disc(Q)}}{2} \right]; \barsgn{a} \right).$$
By the previous lemma, $\barsgn{\det M}=\barsgn{a}$. 
On the other hand, consider an oriented ideal $\mathfrak{I}=\OrelIbasis{\alpha,\beta}{M}$. Then the first coefficient of the quadratic form $\Phi(\mathfrak{I})$ is equal to $\frac{\alpha\overline{\alpha}}{\det M}$, and $\barsgn{\frac{\alpha\overline{\alpha}}{\det M}}=\barsgn{\det M}$, because $\alpha\overline{\alpha}$ is totally positive. Therefore, $\sigma_i(Q)$ is positive definite if and only if $\sigma_i(\det M)>0$.

Let us prove $(ii)\Leftrightarrow(iii)$. Assume that $Q$ is the image under the map $\Phi$ of an oriented ideal $\OrelIbasis{\alpha, \beta}{M}$. Recall that $\Omega=\frac{-w+\sqrt{D_{\Omega}}}{2}$ by (\ref{Omega}), and write $\alpha=a_1+a_2\sqrt{D_\Omega}$, $\beta=b_1+b_2\sqrt{D_\Omega}$ for some $a_1, a_2, b_1, b_2 \in \f{K}$. One can easily compute that
$$\det M = \frac{\cjg{\alpha}\beta-\alpha\cjg{\beta}}{\Omega-\cjg{\Omega}}=2(a_1b_2-a_2b_1),$$ 
and 
$$\frac{\beta}{\alpha}=\frac{a_1b_1-a_2b_2D_\Omega+(a_1b_2-a_2b_1)\sqrt{D_\Omega}}{a_1^2-a_2^2D_\Omega}.$$
Thus, 
$$\Im\left(\frac{\beta}{\alpha}\right)=\frac{a_1b_2-a_2b_1}{a_1^2-a_2^2D_\Omega},$$ 
where $a_1^2-a_2^2D_\Omega$ is totally positive. Hence, 
$$\barsgn{\Im\left(\frac{\beta}{\alpha}\right)}=\barsgn{a_1b_2-a_2b_1}=\barsgn{\det M}. \vspace{-3mm}$$
\end{proof}

The result about totally positive quadratic forms follows immediately from Proposition~\ref{QFsign}.
\begin{corollary}
Let $D_\Omega\prec0$. A quadratic form is totally positive definite if and only if it is the image under the map $\Phi$ of an oriented ideal $\OrelIbasis{\alpha, \beta}{M}$ such that $\Im\left(\frac{\beta}{\alpha}\right)$ is totally positive.
\end{corollary}


\section{Totally imaginary fields}\label{sec:TotImag}

\textcolor[rgb]{1,0,0}{This section (in particular the proof of Proposition~\ref{Prop:NotSoWellDefined}) contains a serious mistake that cannot be corrected; see the appendix of \cite{KZcubes}.}

\bigskip

From the very beginning of this paper, we assumed that the field $K$ has at least one embedding into real numbers. One real embedding is all it takes to be able to distinguish between units $u$ and $-u$: only one of them can be totally positive. This enabled us to fix the square root of $\Disc(Q)$ \emph{uniquely} as a multiple of $\sqrt{D_\Omega}$ by a totally positive unit. In this section, we will extend the correspondence between (classes of) quadratic forms and a group to the case of totally imaginary field $K$.  

\subsection{Motivation}

Let us try to apply the construction from Sections~\ref{Section:Prel} and~\ref{sec:Correspondence} to the case of totally imaginary fields. Since there are no real embeddings, the notion of oriented ideals coincide with usual (fractional) ideals.

\begin{example}
Let $K=\Q(\ii)$ and $\Omega=\frac{1+\sqrt{17}}{2}$, i.e., $L=K(\sqrt{17})$. Consider two quadratic forms:
\[ Q(x,y)=2x^2+5xy+y^2 \qquad \text{ and } \qquad Q'(x,y)=-2x^2-5xy-y^2,\]
both of them of discriminant $17$. Obviously, $Q'(x,y)=-Q(x,y)$. Since $-1\in\UPlusSet{K}$, these two forms are equivalent. Now let ${I}=\Psi(Q)$ and ${I'}=\Psi(Q')$; the question is: Are these ideals equivalent in $\OCl{L/K}$? 

Assume that they are; then there exists $\gamma\in L\setminus\{0\}$ such that $I=(\gamma)\cdot I'$; in such case, $\norm{L/K}{I}=\norm{L/K}{(\gamma)}\cdot\norm{L/K}{I'}$. It is 
\[ {I}=\left[2, \frac{-5+\sqrt{17}}{2}\right] \qquad \text{ and } \qquad {I'}=\left[-2, \frac{5+\sqrt{17}}{2}\right],\]
and so we have $\norm{L/K}{{I}}=(2)=(-2)=\norm{L/K}{{I'}}$. It follows that $\gamma\in\USet{L}$, and hence $I=I'$. Then necessarily $\frac{5+\sqrt{17}}{2}\in I$, so we can find $x,y\in\Z[\ii]$ such that 
\[\frac{5+\sqrt{17}}{2}=2x+\frac{-5+\sqrt{17}}{2}y.\]
Comparing the imaginary parts on both sides, it follows that we can chose $x,y\in\Z$. But then comparing the coefficients by $\sqrt{17}$, we get that $y=1$. Finally, it must hold $4x=10$, which is impossible. Therefore, the ideals $I$ and $I'$ cannot be equivalent.
\end{example}

The example above suggests that the problem is that $Q$ and $-Q$ are equivalent whereas the obtained ideals are not. But the real problem is that the map $\Psi$ is not well defined: The choice of $u\in\UPlusSet{K}$ such that $\sqrt{\Disc(Q)}=u\sqrt{D_\Omega}$ is not unique, as now both $u$ and $-u$ are totally positive. 

We can fix that by choosing $u\in\UPlusSet{K}$ such that $\arg(u)\in[0,\pi)$ where $\arg(u)=\varphi$ such that $u=\abs{u}\cdot(\cos\varphi+\ii\cdot\sin\varphi)$. What happens then with the pair of forms $Q$ and $-Q$? Consider a~quadratic form $Q(x,y)=ax^2+bxy+cy^2$ of discriminant $D\in\DSet$, we have
\[ \Psi(Q)=\left[a,\frac{-b+\sqrt{D}}{2}\right] \qquad \text{and} \qquad \Psi(-Q)=\left[-a,\frac{b+\sqrt{D}}{2}\right]\]
where 
\[ \left[-a,\frac{b+\sqrt{D}}{2}\right]=\left[a,\frac{-b-\sqrt{D}}{2}\right]=\left[\cjg{a},\cjg{\frac{-b+\sqrt{D}}{2}}\right].\]
This means that we would like to make the ideals $[\alpha,\beta]$ and $[\cjg{\alpha},\cjg{\beta}]$ equivalent.

\subsection{Preliminaries}

We fix a totally imaginary number field $K$ and its quadratic extension $L$. We still need $K$ to be of class number one (note that every element of $K$ is totally positive, and hence $h(K)=h^+(K)$), so that every fractional $\OK{L}$-ideal has a two element $\OK{K}$-module basis. As before, for $\alpha\in L$ we denote by $\cjg{\alpha}$ the image of $\alpha$ under the nontrivial element of $\mathrm{Gal}(L/K)$ (thus $\cjg{\alpha}$ is \emph{not} the complex conjugate of $\alpha$). In particular, if $L=K(\sqrt{D_\Omega})$ and $a,b\in K$, then $\cjg{a+b\sqrt{D_\Omega}}=a-b\sqrt{D_\Omega}$.

We will (carefully) use many computations and results from the previous sections.

All the results from Subsection~\ref{Subsec:PrelIdeals} remain true. In particular, $\OK{L}=[1,\Omega]$ and $L=K(\sqrt{D_\Omega})$ with $\sqrt{D_\Omega}=\Omega-\cjg{\Omega}$, and if $L=K(\sqrt{D})$ with $D\in\OK{K}$ fundamental, then $D=u^2D_\Omega$ for some $u\in\USet{K}$.

From Subsection~\ref{subsec:QFs}, we adopt the definition of equivalence of quadratic forms together with the sets $\DSet$ and $\QFSet$. In particular, if $Q$ is a representative of a class in $\QFSet$, then $\Disc(Q)=u^2D_\Omega$ for some $u\in\UPlusSet{K}=\USet{K}$. 

\begin{example}[Cf. Remark~\ref{rem:equivQF}]
We want to show that in the definition of equivalent quadratic forms, units other than $1$ are important.

Let $K=\Q(\ii)$, and consider the quadratic forms $Q(x,y)=x^2+4xy+2y^2$ and $Q'(x,y)=\ii x^2+4xy-2\ii y^2$. Clearly $Q'(x,y)=-\ii Q(\ii x, y)$, but there is no matrix of determinant $1$ which would provide the equivalence between $Q$ and $Q'$. That can be seen as follows: by the first equality of (\ref{Eq:equiv_QF-koef}), we need to find $p,r\in\Z[\ii]$ such that $p^2+4pr+2r^2=\ii$. But that is not possible, because the imaginary part of $p^2+4pr+2r^2$ is divisible by 2.
\end{example}

We need to rectify the ambiguity of the square root of $\Disc(Q)$: Set $\sqrt{\Disc(Q)}=u\sqrt{D_\Omega}$ for $u\in\USet{K}$ such that $\arg(u)\in[0,\pi)$ (where $\arg(u)=\varphi$ such that $u=\abs{u}\cdot(\cos\varphi+\ii\cdot\sin\varphi)$). Note that this condition hold for exactly one from the pair $u$ and $-u$, and hence this choice is always possible and unique.

We will also need Lemma~\ref{Lemma:roots_of_equiv_QFs}; since it does not hold in the original formulation, we restate it here.

\begin{lemma} \label{Lemma:roots_of_equiv_QFs-IMAG}
Let $Q(x,y)=ax^2+bxy+cy^2$ be a quadratic form, and let $p, q, r, s \in\OK{K}$ be such that $ps-qr\in\UPlusSet{K}$. Consider the quadratic form $\widetilde{Q}(x,y)=Q(px+qy, rx+sy)=\widetilde{a}x^2+\widetilde{b}xy+\widetilde{c}y^2$ equivalent to $Q(x,y)$. Denote $D=\Disc(Q)$ and $\widetilde{D}=\Disc(\widetilde{Q})$. Then it holds 
\[
\frac{p\frac{-\widetilde{b}+\sqrt{\widetilde{D}}}{2\widetilde{a}}+q}{r\frac{-\widetilde{b}+\sqrt{\widetilde{D}}}{2\widetilde{a}}+s}
=
\left\{\begin{array}{ll}
\frac{-b+\sqrt{D}}{2a},& \text{ if } \arg(ps-qr)\in[0,\pi),\vspace{2mm}\\
\frac{-b-\sqrt{D}}{2a},& \text{ if } \arg(ps-qr)\in[\pi,2\pi).
\end{array}\right.
\]
\end{lemma}

\begin{proof}
The only difference from the proof of Lemma~\ref{Lemma:roots_of_equiv_QFs} is at its end, where we now have to distinguish between two possibilities: $\sqrt{\widetilde{D}}=(ps-qr)\sqrt{D}$ if $\arg(ps-qr)\in[0,\pi)$, and $\sqrt{\widetilde{D}}=-(ps-qr)\sqrt{D}$ if $\arg(ps-qr)\in[\pi,2\pi)$. 
\end{proof}

\subsection{Yet Another Class Group}\label{subsec:ClImag}

As we have seen in the motivation above, we have to change the group we work with in the desired correspondence.

Recall that $\ISet{L}$ denotes the group of all $\OK{L}$-ideals and $\PSet{L}$ ist subgroup of principal ideals (where all ideals are fractional). If $I$ is an $\OK{L}$-ideal, then $I\PSet{L}$ stands for its class in $\Cl{L}$. We set $\cjg{I}=\{\cjg{\alpha}~|~\alpha\in I\}$. Obviously, $\alpha\in I$ if and only if $\cjg{\alpha}\in\cjg{I}$, and hence $\cjg{I}\cdot\cjg{J}=\cjg{IJ}$. If $I=[\alpha,\beta]$, then $\cjg{I}=[\cjg{\alpha}, \cjg{\beta}]$. Moreover, $\cjg{(\gamma)}=\bigl(\cjg{\gamma}\bigr)$ for any $\gamma\in L$; thus, $\cjg{I\PSet{L}}=\cjg{I}\PSet{L}$.

We define a subgroup $\Gr{L/K}$ of $\Cl{L}$:
\[\Gr{L/K}=\left\{I\bigl(\cjg{I}\bigr)^{-1}\PSet{L}~|~I\in\ISet{L}\right\}.\]
This is indeed a group: If $I\bigl(\cjg{I}\bigr)^{-1}\PSet{L}, J\bigl(\cjg{J}\bigr)^{-1}\PSet{L}\in\Gr{L/K}$, then 
\[\left(I\bigl(\cjg{I}\bigr)^{-1}\PSet{L}\right)\cdot\left(J\bigl(\cjg{J}\bigr)^{-1}\PSet{L}\right)=(IJ)\bigl(\cjg{IJ}\bigr)^{-1}\PSet{L}\in\Gr{L/K}.\]
Thus, we can define a group which will play the role of the relative oriented class group.

\begin{definition}
Let $K$ be a totally imaginary field and $L$ its quadratic extension. We set
\[\ClImag{L/K}=\quotient{\Cl{L}}{\Gr{L/K}}\]
and call it the \emph{imaginary class group}. If two ideals $I$, $J$ represent the same class of $\ClImag{L/K}$, we write again $I\sim J$ and say that they are \emph{equivalent}.
\end{definition}

Note that in $\ClImag{L/K}$ holds $I\sim\cjg{I}$, so in particular $[\alpha,\beta]\sim[\cjg{\alpha},\cjg{\beta}]$. By the same computation as in the proof of Proposition~\ref{Prop:InverseIdeals}, we have that $[\alpha,\beta]^{-1}\sim [\cjg{\alpha},-\cjg{\beta}]=[\cjg{\alpha},\cjg{\beta}]$; therefore, $[\alpha,\beta]\sim [\alpha,\beta]^{-1}$, and so every nontrivial element of the imaginary class group has order two.

\subsection{The Correspondence}
We define maps $\Phi'$ and $\Psi'$ analogous to the maps $\Phi$ and $\Psi$ from Section~\ref{sec:Correspondence}. As before, we omit to write the classes to simplify the notation.
\[\begin{array}{cccc}
\Phi': & \ClImag{L/K} & \longrightarrow & \QFSet  \vspace{2mm}	\\
 & [\alpha,\beta] & \longmapsto & \frac{1}{\det M} \norm{L/K}{\alpha x-\beta y }=\frac{\alpha\cjg{\alpha}x^2-(\cjg{\alpha}\beta+\alpha\cjg{\beta})xy+\beta\cjg{\beta}y^2}{\det M}\vspace{6mm}\\
\Psi': & \QFSet  & \longrightarrow & \ClImag{L/K} \vspace{2mm}	\\
 & Q(x,y)=ax^2+bxy+cy^2 & \longmapsto & \left[a, \frac{-b+\sqrt{\Disc(Q)}}{2}\right]\\
\end{array}\]
where $\det M=\frac{\cjg\alpha\beta-\alpha\cjg\beta}{\Omega-\cjg\Omega}$ and we set $\sqrt{\Disc(Q)}=u\sqrt{D_\Omega}$ for some $u\in\USet{K}$ such that $\arg(u)\in[0,\pi)$. 

\begin{proposition} \label{Prop:NotSoWellDefined}
The maps $\Phi'$ and $\Psi'$ are well defined.
\end{proposition}

\begin{proof}
The proof that $\Phi'([\alpha,\beta])$ lies in $\QFSet$ is exactly the same as the proof of Proposition~\ref{Prop:Phi_Im}. To prove that $\Phi'$ does not depend on the choice of the representative, we proceed similar as in the proof of Proposition~\ref{Prop:Phi_repr}: First, let $[\alpha, \beta]=[p\alpha+q\beta, r\alpha+s\beta]$ for some $p,q,r,s\in\OK{K}$ such that $ps-qr\in\USet{K}$. Set $Q(x,y)= \Phi'([\alpha,\beta])$ and $\widetilde{Q}(x,y)=\Phi'([p\alpha+q\beta, r\alpha+s\beta])$; as before, we get 
\[\widetilde{Q}(x,y)=\frac{1}{ps-qr}Q(px-qy,-rx+sy).\]
But since now we have $\UPlusSet{K}=\USet{K}$, the equivalence of $Q(x,y)$ and $\widetilde{Q}(x,y)$ follows immediately. For $\gamma\in L\setminus\{0\}$, we get $\Phi'([\gamma\alpha,\gamma\beta])=\Phi'([\alpha,\beta])$ exactly as before. Finally, it is easy to check that $\Phi'([\alpha,\beta])=-\Phi'([\cjg\alpha,\cjg\beta])$ (where the minus sign comes from $\det M$), and since $-1\in\UPlusSet{K}$, there is nothing to prove.

\smallskip

To prove that $\Psi'$ is well-defined, first of all note that the value of $\sqrt{\Disc(Q)}$ is uniquely given (because the value of $\sqrt{D_\Omega}$ is fixed). The proof that $\Psi'(Q)\in\ClImag{L/K}$ is the same as the proof of Proposition~\ref{Prop:Psi_Im}; in fact, it is even easier, because now we are not interested in the orientation. So it remains to prove that $\Psi'(Q)$ does not depend on the choice of the representative in $\QFSet$. We will proceed similar as in the proof of Proposition~\ref{Prop:Psi_repr}.

Let $v\in\UPlusSet{K}$; to compute $\Psi'(vQ)$, we have to distinguish two cases: If $\arg(v)\in[0,\pi)$, then $\sqrt{\Disc(vQ)}=v\sqrt{\Disc(Q)}$ and the proof is the same as before. On the other hand, if $\arg(v)\in[\pi,2\pi)$, then $\sqrt{\Disc(vQ)}=-v\sqrt{\Disc(Q)}$; in such case 
\[\Psi'(Q)=\left[a, \frac{-b+\sqrt{\Disc{Q}}}{2}\right] \qquad \text{and} \qquad \Psi'(vQ)=\left[a, \frac{-b-\sqrt{\Disc{Q}}}{2}\right].\]
But the second ideal is the conjugate of the first one, so the two ideals are equivalent in $\ClImag{L/K}$.

Finally, we have to compare the images of the forms $Q(x,y)$ and $Q(px+qy,rx+sy)$ by $\Psi'$, where $p,q,r,s\in\OK{K}$ are such that $ps-qr\in\UPlusSet{K}$. Again, we proceed as in the proof of Proposition~\ref{Prop:Psi_repr}, only now we use Lemma~\ref{Lemma:roots_of_equiv_QFs-IMAG} instead of Lemma~\ref{Lemma:roots_of_equiv_QFs}; thus, we get two possibilities, to which ideal $\Psi'(Q(px+qy,rx+sy))$ is equivalent, depending on $\arg(ps-qr)$: either $\left[a,\frac{-b+\sqrt{\Disc(Q)}}{2}\right]$ or $\left[a,\frac{-b-\sqrt{\Disc(Q)}}{2}\right]$. But these two ideals are equivalent in $\ClImag{L/K}$, so we are done.
\end{proof}

Finally, we can state the desired version of Theorem~\ref{Theorem:Bijection} and Corollary~\ref{Cor:GroupStructure}. 

\begin{theorem}\label{Theorem:Bijection-Imag}
The maps $\Phi'$ and $\Psi'$ are mutually inverse bijections between $\QFSet$ and $\ClImag{L/K}$ and provide a group structure on $\QFSet$. The identity element of this group is represented by the form $x^2-(\Omega+\cjg\Omega)xy+\Omega\cjg\Omega y^2$ and each element is inverse to itself.
\end{theorem}

\begin{proof}
The only difference to the proves of Theorem~\ref{Theorem:Bijection} and Corollary~\ref{Cor:GroupStructure} is in the part about inverse elements; it has been explained at the end of Subsection~\ref{subsec:ClImag} that each element of the group $\ClImag{L/K}$ is of order one or two, and thus inverse to itself; the statement follows.  
\end{proof}

\section*{Acknowledgment}
 I wish to thank V\'{\i}t\v{e}zslav Kala for his excellent guidance and useful suggestions. \\
I also want to thank Professor Rainer Schulze-Pillot for pointing out references \cite{Kneser, Towber}. Finally, I would like to express my thanks to the unknown referee, who pointed out the problem with totally imaginary fields.


\end{document}